\documentclass[makeidx,fullpage]{article}
\usepackage[utf8]{inputenc}
\usepackage[T1]{fontenc}
\usepackage{amsmath}
\usepackage{amsthm}
\usepackage{amssymb}
\usepackage{mathrsfs}
\usepackage{dsfont}
\usepackage{stmaryrd}
\usepackage{amscd}
\usepackage{float}
\usepackage[a4paper]{geometry}
\usepackage{graphicx}
\usepackage{subfigure}
\usepackage{epsfig}
\usepackage[multiple]{footmisc}
\usepackage{lmodern}
\usepackage{pstricks}
\usepackage[colorlinks=true]{hyperref} 
\usepackage{cleveref}
\usepackage{verbatim}
\usepackage[colorinlistoftodos,french]{todonotes}
\usepackage[all]{xy}
\usepackage{caption}

\usepackage{array}

\usepackage[normalem]{ulem}

\title{\sf Optimal shape of stellarators for magnetic confinement fusion}
\author{ Yannick Privat\footnote{IRMA, Universit\'e de Strasbourg, CNRS UMR 7501, Inria, 7 rue Ren\'e Descartes, 67084 Strasbourg, France ({\tt yannick.privat@unistra.fr}).}~\footnote{Institut Universitaire de France (IUF).}
	\and R\'emi Robin\footnote{Laboratoire Jacques-Louis Lions, Sorbonne Universit\'e, Paris, France ({\tt remi.robin@inria.fr}).}
	\and Mario Sigalotti\footnote{Inria, France ({\tt mario.sigalotti@inria.fr}).}
}

\newcommand{\R}{\mathbb{R}}
\newcommand{\N}{\mathbb{N}}
\newcommand{\Z}{\mathbb{Z}}

\newcommand{\eps}{\varepsilon}
\newcommand{\F}{\mathcal{F}}

\newcommand{\rmin}{r_{\rm{min}}}
\newcommand{\Oad}{\mathscr{O}_{\rm ad}}

\DeclareMathOperator{\dive}{div}
\DeclareMathOperator{\curl}{curl}
\newcommand{\grad}{\nabla}
\DeclareMathOperator{\Ima}{Im}

\makeatletter
\DeclareFontFamily{U}{tipa}{}
\DeclareFontShape{U}{tipa}{m}{n}{<->tipa10}{}
\newcommand{\arc@char}{{\usefont{U}{tipa}{m}{n}\symbol{62}}}%

\newcommand{\arc}[1]{\mathpalette\arc@arc{#1}}

\newcommand{\arc@arc}[2]{%
  \sbox0{$\m@th#1#2$}%
  \vbox{
    \hbox{\resizebox{\wd0}{\height}{\arc@char}}
    \nointerlineskip
    \box0
  }%
}
\makeatother

\renewcommand{\geq}{\geqslant}
\renewcommand{\leq}{\leqslant}

\newtheorem{theorem}{Theorem}

\newtheorem{proposition}{Proposition}

\newtheorem{definition}{Definition}
\newtheorem{lemma}{Lemma}
\theoremstyle{definition}
\theoremstyle{definition}\newtheorem{remark}{Remark}

\newcommand{\ms}[1]{{\color{magenta} #1}}

\begin{document}
\maketitle

\begin{abstract}
We are interested in the design of stellarators, devices for the production of controlled nuclear fusion reactions alternative to tokamaks. The confinement of the plasma is entirely achieved by a helical magnetic field created by the complex arrangement of coils fed by high currents around a toroidal domain. Such coils describe a surface called ``coil winding surface'' (CWS). In this paper, we model the design of the CWS as a shape optimization problem, so that the cost functional reflects both optimal plasma confinement properties, through a least square discrepancy, and also manufacturability, thanks to geometrical terms involving the lateral surface or the curvature of the CWS. 

We completely analyze the resulting problem: on the one hand, we establish the existence of an optimal shape, prove the shape differentiability of the criterion, and provide the expression of the differential in a workable form. 
On the other hand, we propose a numerical method and perform simulations of optimal stellarator shapes. We discuss the efficiency of our approach with respect to the literature in this area.
\end{abstract}

\noindent\textbf{Keywords:} shape optimization, plasma Physics, Biot and Savart operator, Riemannian manifolds.

\medskip

\noindent\textbf{AMS classification (MSC 2020):} 49Q10, 49Q12, 78A25, 65T40.

\section{Introduction}

\subsection{Motivations: towards a shape optimization problem}
Nuclear fusion is a nuclear reaction involving the use of light nuclei.
In order to produce energy by nuclear fusion, high temperature plasmas\footnote{This is a particular state of matter when it becomes totally \emph{ionized}, i.e., when all its atoms have lost one or more peripheral electrons. This is the most common state of matter in the universe because it is found (at 99\%) in the stars, the interstellar medium, and earth's ionosphere.} must be produced and confined. For these reactions to occur, the nuclei must get close to each other at very small distances. They must therefore overcome the Coulomb repulsion. This happens naturally in a plasma during collisions if the energy of the nuclei is sufficient. This is the objective of devices called tokamaks, steel magnetic confinement chambers that allow a plasma to be controlled in order to study and experiment with energy production by nuclear fusion. The magnetic confinement technique allows to maintain a sufficient temperature and density of the plasma, in an intense magnetic field. 
The simplest configuration 
for the magnetic field
is the toroidal solenoid; this is the configuration found in most current experiments.

Unfortunately, the magnetic field is not uniform in general, which causes a vertical drift of the particles, in opposite directions for the ions and for the electrons. This charge separation creates a vertical electric field which, in turn, causes the particles to drift out of the torus. This phenomenon dramatically reduces the confinement. To get around this obstacle, the effect of such drifts is canceled   by giving a poloidal component\footnote{The terms toroidal and poloidal refer to directions relative to a torus of reference. The poloidal direction follows a small circular ring around the surface, while the toroidal direction follows a large circular ring around the torus, encircling the central void. } to the magnetic field: the field lines are wound on nested toroids. Thus, the particles, following the magnetic field lines, have their vertical drift cancelled at each turn.
In a tokamak, the poloidal magnetic field is created by a toroidal electric current circulating in the plasma. This current is called \emph{plasma current}. 

A possible alternative to correct the problems of drift of magnetically confined plasma particles in a torus is to modify the toroidal shape of the device, by breaking the axisymmetry, yielding to the concept of stellarator.
A stellarator is analogous to a tokamak except that it does not use a toroidal current flowing inside the plasma to confine it. The poloidal magnetic field is generated by external coils, or by a deformation of the coils responsible for the toroidal magnetic field. This system has the advantage of not requiring plasma current and therefore of being able to operate continuously; but  it comes at the cost of more complex 
coils (non-planar coils) and of a more important neoclassical transport \cite{Helander-book}. 

The confinement of the plasma is then entirely achieved by a helical magnetic field created by the complex arrangement of coils around the torus, supplied with strong currents and called poloidal coils. 

Despite the promise of very stable steady-state fusion plasmas, stellarator technology also presents significant challenges related to the complex arrangement of magnetic field coils. These magnetic field coils are particularly expensive and especially difficult to design and fabricate due to the complexity of their spatial arrangement. 

In this paper, we are interested in the search for the optimal shape of stellarators, i.e., the best coil arrangement (provided that it exists) to confine the plasma. In general, two steps are considered: first, the shape of the plasma boundary is determined in order to optimize the physical properties, among which the neoclassical transport and the magnetohydrodynamic (MHD) stability. In a second step, we search for the coil shapes producing approximately the ``target" plasma shape resulting from the previous step.

In this article, we focus entirely on the second step, assuming that the target magnetic field $B_T$ is known. It is then convenient to define a coil winding surface (CWS) on which the coils will be located (see \Cref{fig:CWS}). The optimal arrangement of stellarator coils corresponds then to the determination of a closed surface (the CWS) chosen to guarantee that the magnetic field created by the coils is as close as possible to the target magnetic field $B_T$. Of course, it is necessary to consider feasibility and manufacturability constraints. We will propose and study several relevant choices of such constraints in what follows.

    \begin{figure}[h!]
        \begin{center}
        \includegraphics[width=0.5\textwidth]{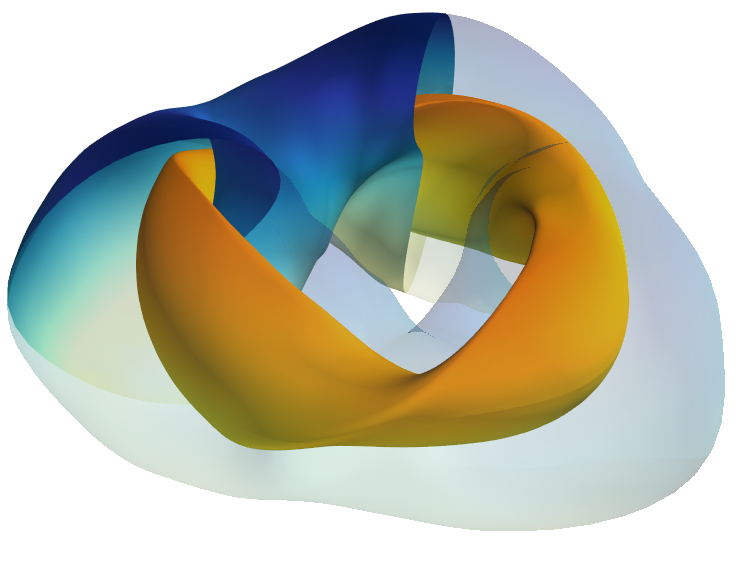}
        \caption{CWS (blue and white) and plasma surface (orange) of the National Compact Stellarator Experiment (NCSX) designed by the Princeton Plasma Physics Laboratory. There is a three-folds discrete symmetry in the design.\label{fig:CWS}}
    \end{center}
    \end{figure}

\subsection{State of the art and main contributions of this article}
The question of determining the best location of coils around a stellarator, reformulated as an optimal surface problem, is a major issue for the construction of stellarators with efficient confinement properties. The physical and mathematical literature dedicated to plasmas is rich of references on this issue. We mention hereafter a non-exhaustive list of various important contributions around this problem.
Let us first mention \cite{imbert-gerardIntroductionSymmetriesStellarators2019}, where all the basic theoretical elements to understand the modeling of stellarator magnetic fields are gathered.

Regarding optimal design issues, let us distinguish between several optimization/optimal control approaches and modeling choices. Each discrete stellarator coil can be represented as a closed one-dimensional curve embedded in $\R^3$ \cite{Zhu_2017,Zhu_2018, Zhu_2018b}. In these references, several optimization methods are tested among which the steepest descent and Newton like methods. 

Another common choice consists in using the aforementioned CWS, in other words to define a closed toroidal winding surface enclosing the plasma surface on which all coils lie. Two kinds of issues related to the optimal design of stellarators can then be addressed. The simplest is to assume the CWS to be given, and to look for currents  on  this  surface generating the desired magnetic field for confining the plasma. Indeed, in the limit of a large
number of coils, a set of discrete coils can be described by a continuous current density on the CWS.
Let us mention NESCOIL \cite{merkelSolutionStellaratorBoundary1987,Pomphrey_2001}, where the current potential representing a surface current distribution is sought such that the normal component of the magnetic field vanishes in a least-squares sense at the plasma boundary. In the same vein, REGCOIL \cite{landremanImprovedCurrentPotential2017} improves the NESCOIL approach by adding a Tikhonov regularization term in the minimization functional whereas COILOPT \cite{doi:10.13182/FST02-A206} uses an explicit representation of modular coils on a toroidal winding surface. A review of such approaches can be found in \cite{Gates_2017}. Recently, a similar problem where an extra Laplace forces penalization term is taken into account has been investigated in \cite{Robin-Volpe}.

A much more difficult problem is to determine the CWS and the density current distribution at the same time. This is expected to improve the performances of the resulting device. On the other hand, this approach requires solving a dual optimization problem, including a rather challenging surface optimization problem. This is 
the main purpose of this article. In the following we mention some of the many contributions on this topic and position our contribution through this literature. In \cite{Paul_2018}, this problem is modeled using a cost functional written as the weighted sum of four terms: the first one is the surface-integrated-squared  normal  magnetic field on the desired plasma surface. The second is the opposite of the total volume enclosed by the 
coil-winding surface, acting to enforce the coil-plasma separation. The third one  is a measure of  the  spectral  width of the Fourier series describing the coil-winding surface. This allows to overcome the non-uniqueness of the Fourier series representation of the coil-winding surface. The last one is the $L^2$ norm  of  the  current  density, allowing to  obtain  coils  with  good  manufacturing properties. It is important to note here, and this is related to the motivation for this paper, that the approach developed in \cite{Paul_2018} rests upon a (truncated) Fourier series parameterization of the surface equation. The authors thus compute derivatives of their cost with respect to these Fourier coefficients. 

In \cite{paul_abel_landreman_dorland_2019}, a more complex model involving a drift kinetic equation is considered and similar shape optimization issues are investigated. 

In what follows, we propose a continuous approach, which does not rely on any parameterization of the surfaces involved. We use the notion of Hadamard variation and shape derivative. We rigorously analyze, in a continuous framework,  the sensitivity with respect to the domain
of a REGCOIL-type cost.  
We thus obtain intrinsic expressions with respect to any parametrization.  This makes our approach flexible and the formulas obtained by using developments of the parametric equation of the surfaces in Fourier series can be adapted without any difficulty to other choices of parametrization.  We also propose several choices of manufacturability terms in the cost functional and discuss their relevance.

\medskip

The issues addressed in the following as well as our main contributions are summed-up hereafter:
\begin{itemize}
\item \textbf{Modeling of the problem (Section~\ref{sec:model}).} Using the CWS concept, we propose a continuous formulation of the question of the best coil arrangement as a shape optimization problem, regardless of any surface parametrization. In particular, several choices of manufacturing constraints are proposed. They are integrated to the cost functional using a penalization/regularization term. From the mathematical point of view, the main issue comes to minimize a functional involving the trace of the solution of an elliptic partial differential equation (PDE) on a manifold, under geometrical constraints involving the 
distance 
to this manifold. 
\item \textbf{Analysis of the shape optimization problem (Sections~\ref{sec:analyseExist} and \ref{sec:diffSOP}).} Having in mind the determination of an efficient algorithm for finding an optimal form for the above problem, we focus mainly on two questions. The first one is dedicated to the existence of an optimal shape (Section~\ref{sec:existence_SOP}). In this context, the developed approach is not completely standard and requires to carefully establish semicontinuity properties of the trace of the solution of the PDE on manifolds satisying a uniform regularity property. The second one concerns the establishment of optimality conditions using the notion of form derivative (Section~\ref{sec:shapeD}). Here again, due to the particular nature of the PDE at stake, the classical approach cannot be used in a direct way and many adaptations are necessary. We establish a workable expression of this derivative, which is the basis of the numerical approaches developed in the next section. 
\item \textbf{Numerical implementation (Section~\ref{sec:num}).} A relevant aspect of this paper is that the study of the sensitivity of the studied criterion to a variation of the shape of the stellarator is carried out without using any parameterization of the surface to be designed. As a result, the sensitivity relations 
obtained at the end of the previous step are totally intrinsic with respect to any parameterization of the surface. 
As a consequence, we can apply the more robust ``optimize then discretize" approach, 
instead of a ``discretize then optimize" procedure as in most of the methods implemented for this application.
The shape derivatives
constitute the basis of a quasi-Newton optimization method that we implement by using a parametric representation of the surface in terms of Fourier series.  
\end{itemize}

\subsection{Notations}\label{sec:notations}

In what follows, the notation $S$ is used to denote a $\mathscr{C}^{1,1}$ toroidal surface\footnote{By \emph{toroidal surface}, we mean here the range of the toroidal solenoid by a homeomorphism. In what follows, we will rather consider smooth toroidal surfaces, where the wording ``smooth" refers to at least $\mathscr{C}^{1,1}$ regularity.\label{footn:toroidalM}} in $\R^3$, equipped with the Riemannian metric 
induced by the canonical embedding $i_S:S \hookrightarrow \R^3$, {i.e., the scalar product between two vectors $v$ and $w$ tangent to $S$ at a common point is equal to 
$\langle v,w\rangle$, where $\langle \cdot,\cdot \rangle$ denotes the Euclidean scalar product in $\R^3$}.  
We also denote by $\mu_S$ the associated Riemannian volume form, which coincides with the two-dimensional Hausdorff measure on $S$ (\cite[Theorem~2.10.10]{federer} and \cite[Theorem~2 in Section~2.2]{MR1158660}). 
We write $V$ to denote the bounded domain of $\R^3$ such that $S=\partial V$.

Throughout this article, we use the following notation:
\begin{itemize}
\item 
For $n\in \N^*=\{1,2,\dots\}$ and any integer $m\ge n$, $\mathscr{H}^n$ denotes the $n$-dimensional Hausdorff measure in $\R^m$;
\item $P$ denotes a smooth 
toroidal domain\footnote{\emph{
toroidal domain} stands for any three-dimensional domain whose boundary is a toroidal surface} of $\R^3$ standing for the plasma domain;
\item $\mathfrak{X} (S)$ denotes the set of smooth tangent vector fields on $S$;
\item 
$\mathscr{F}_S=L^2(\Gamma(TS))$ denotes the completion of $\mathfrak{X} (S)$  for the inner product
$$
\forall (X_1,X_2)\in \mathfrak{X} (S)^2, \qquad \langle X_1 , X_2 \rangle_{\mathscr{F}_S} =\int_S \langle X_1, X_2 \rangle
d\mu_S;
$$
\item $\times$ denotes the cross product in $\R^3$;
\item given a function $F:\R^{n_1}\to \R^{n_2}$ and $x\in \R^{n_1}$, $DF(x)$ denotes the $n_1\times n_2$ Jacobian matrix of $F$. In the case where $n_1=n_2$, $|DF|$ stands for the absolute value of the determinant of $F$. The symbol $D_x$ is used to denote the Jacobian operator with respect to the (vector) variable $x$;
\item 
$\nabla_{S}$ denotes the \emph{tangential gradient} to $S$ in $\R^3$, defined for every differentiable function $f:\R^3\to\R$ by $\nabla_{S}f=\nabla f-\langle \nabla f , \nu \rangle \nu$ on $S$, where $\nu$ stands for the outward normal vector to $V$.  Similarly, the notation $\operatorname{div}_S$ stands for the \emph{tangential divergence} given by $\operatorname{div}_S\theta=\operatorname{div}\theta- \langle D\theta\nu , \nu \rangle$ on $S$, where 
$\theta$ is a vector field on $\R^3$;
\item The norm on $\mathfrak{X}(S)$ induced by the inner product $\langle\cdot , \cdot\rangle_{\mathscr{F}_S}$ is denoted $\Vert \cdot\Vert_{\mathscr{F}_S}$;
\item $\mathscr{F}_S^0$ is the closure under the norm $\Vert\cdot\Vert_{\mathscr{F}_S}$ of the (tangential) divergence-free vectors of $\mathfrak{X}(S)$;
\item The flat two-dimensional torus is denoted by $T=(\R/\Z)^2$.
$\mathfrak{X}(T)$, $\mathscr{F}_T$ and $\mathscr{F}_T^0$ are defined similarly to what has been done above;
\item 
The Hausdorff distance  $d_V$ and the signed distance $b_V$  from $V$ are defined as:
    $$d_V(x)=\inf_{y\in V} |x-y|, \qquad  b_V(x)=d_V(x) - d_{\R^3 \setminus V}(x);$$
\item If $h>0$, the \emph{$h$-tubular neighborhood} $U_h(V)$ of $V$ is the level set 
    $$ U_h(V)=\{ x\in \R^3 \mid d_V(x)<h\}$$
    of $d_V$;
\item The \emph{reach of $V$} \cite{federer}   is given by
    $$
    \operatorname{Reach}(V)=\sup \{h>0 \mid d_V \text{ is differentiable on }U_h(V)\setminus \bar{V}\}.
    $$
 More explanations are provided in Appendix~\ref{sec:append:reach}. For more exhaustive informations about this notion, we refer to \cite{federer} and \cite[Sect. 6.6]{delfour_shapes_2011};
\item Given a differentiable vector field
 $\theta:\R^3\to \R^3$, we denote by 
 $e(\theta)$
the symmetric part of the Jacobian matrix $D\theta$, that is,
\begin{equation}\label{eq:*}
e(\theta)=D\theta+(D\theta)^T;
\end{equation}
\item For two Banach spaces $E$ and $F$, we denote by $\mathcal{L}(E,F)$ the Banach space of continuous linear maps from $E$ to $F$ and by $\mathcal{L}(E)$ the Banach space of continuous endomorphisms;
\item the adjoint of a linear operator $L$ is denoted by $L^\dagger$;
\item If $A$ and $B$ denote two matrices in $\mathbb{M}_3(\R)$, we define their {\it doubly contracted product} as
$$
A:B=\sum_{i,j=1}^3A_{ij}B_{ij};
$$
\item $I_3$  denotes the identity matrix in $\R^3$.
\end{itemize}
\subsection{Modeling: towards a shape optimization problem}\label{sec:model}
Since we are interested in solving a shape optimization problem whose unknown is the coil winding surface $S$, we are led to make some 
assumptions on $S$ motivated by the application under consideration. 
In particular, we  assume in what follows that the distance $d(S,P)$ between $S$ and the plasma domain $P$  is uniformly bounded from below, namely,
we fix $\delta>0$ and we require that 
\begin{equation}\label{H1} \tag{$\mathscr{H}_{{\rm dist},P,\delta}$}
   d(S,P)= \inf_{x\in S,y\in P} |x-y|=\inf_{x\in S}d_P(x) \geq \delta.
\end{equation}

We now introduce the main operator we will deal with, which plays a crucial role in electromagnetism: the so-called {\it Biot and Savart} operator. This operator associates with each current distribution on  $S$ the corresponding magnetic field in $P$. 
It can be considered as a kind of inverse of the curl operator.

\begin{definition}[The Biot and Savart operator $\operatorname{BS}_S$ \cite{ENCISO201885}]\label{def:opBS}
Let $S$ be a smooth two-dimensional manifold and $X$ belong to $\mathscr{F}_S$. 
Let $\delta_S$ denote the single layer distribution supported on $S$ defined by
$$\forall \varphi \in \mathscr{C}^\infty_c(\R^3,\R^3),\qquad  \langle X \delta_S, \varphi \rangle = \int_S  \langle\varphi, X\rangle d\mu_S.$$
Let $u$ denote the unique distributional solution of the PDE
$$
\left\{\begin{array}{ll}
\nabla \times u= X \delta_S & \text{in }\mathscr{D'}(\R^3)\\
\langle\nabla, u\rangle=0 & 
\end{array}\right.
$$
that falls off at infinity, i.e.,
\[u(y)=\int_S \frac{(x-y)\times X(x)}{|x-y|^3} d\mu_S(x),\qquad y\in \R^3\setminus S.\]
Then the {\it Biot and Savart} operator 
is defined as the map 
$\operatorname{BS}_S:\mathscr{F}_S  \longrightarrow  L^2(P, \R^3)$
associating with $X$ the restriction of $u$ to the plasma domain $P$.
By introducing the kernel $K$ given by 
\begin{align*}
    K : [\R^3]^2\backslash \{(x,x)\mid x \in \R^3\}  &\longrightarrow \R^3 
    \\
 (x,y) &\longmapsto  
 \frac{x-y}{|x-y|^3}, 
\end{align*}
one has
\begin{equation}\label{eq:BS-with-kernel}
\operatorname{BS}_S(X)(y)=\int_S K(x,y)\times X(x) d\mu_S(x),\qquad y\in P.
\end{equation}
\end{definition}

\begin{remark}
According to \eqref{H1}, the restriction of $K$ to $S\times P$ is uniformly bounded.  
By standard regularity results for parameterized integrals, the mapping $P\ni y \mapsto \operatorname{BS}_S(X)(y)$ is smooth and the operator $\operatorname{BS}_S$, seen as going from $\mathscr{F}_S$ to $\mathscr{C}^k(P,\R^3)$ with $k\in \N\cup \{+\infty\}$, is continuous.
As a consequence, the operator $\operatorname{BS}_S:\mathscr{F}_S \to L^2(P,\R^3)$ is compact.

In what follows, we will use several times that for every $x,y,h\in \R^3$ with $x\ne y$ we have
\begin{align}
    \label{eq:der_K}
  D_x K(x,y)(h)
    =\lim_{\varepsilon\searrow 0}\frac{K(x+\varepsilon h,y)-K(x,y)}{\varepsilon} 
    =
     \frac{h}{|x-y|^3}- \frac{3\langle (x-y), h \rangle (x-y)}{|x-y|^5}.
\end{align}
\end{remark}

\paragraph{Computation of the optimal current $\boldsymbol{j}$.}
In view of modeling the optimal design problem we will deal with, let us now introduce a target magnetic field $B_T \in L^2(P, \R^3)$.

The target magnetic field $B_T$ being given, we model the optimal design of a stellarator problem as a kind of regularized least square problem, where one aims at determining  both the current $j$ and the manifold shape $S$ leading to the  magnetic field closest to $B_T$ on $S$.
To this aim, and according to the REGCOIL procedure \cite{landremanImprovedCurrentPotential2017}, we introduce the shape functional $C$ defined, for every closed smooth two-dimensional manifold $S$, as 
\begin{equation}\tag{$\mathscr{P}_{S}$}\label{defpb:PS}
 \boxed{C(S)= \inf_{j\in \mathscr{F}_S^0} \|\operatorname{BS}_S j -B_T \|^2_{L^2(P,\R^3)} + \lambda \|j\|^2_{\mathscr{F}_S}, }
\end{equation}
where $\lambda>0$ denotes a regularization parameter.

\paragraph{The shape optimization problem.}

To state the shape optimization problem that we will consider, let us first define the set of admissible manifolds. We gather hereafter several conditions evoked previously 
that we will take into account in the search of the CWS. 

\begin{itemize}
   \item {\it Topology and uniform boundedness.} To preserve the topology of the device (see Footnote~\ref{footn:toroidalM}), we will only consider CWSs that are two-dimensional closed toroidal manifolds. Moreover, we will fix a compact set $D$ of $\R^3$ and we require the CWS to be contained in $D$. 
        \item {\it Uniform distance constraint of the coils to the plasma}. To build the vacuum vessel around the plasma, we will assume that the CWS satisfies assumption \eqref{H1}.
    \item \textit{Manufacturing cost}. In order to avoid irregular shapes that are too difficult to build, we will assume that the CWS has a minimal regularity, say $\mathscr{C}^{1,1}$, 
    and a minimal reach condition.
    More precisely, we will assume that the reach
    of the CWS is uniformly bounded from below by some $r_{\min}>0$. We recall that this condition imposes that the curvature radii (where they can be defined) are larger than $r_{\min}$ and that there is no bottleneck of distance smaller than $2r_{\min}$ (see, e.g., \cite[Figure 3]{aamari_estimating_2019}).
To sum-up,    
\begin{equation}\tag{$\mathscr{H}_{{\rm reach},r_{\min}}$}\label{Hyp_reach}
\text{$S$ is a $\mathscr{C}^{1,1}$ closed toroidal surface such that }\operatorname{Reach}(S)\geq r_{\min}>0.
\end{equation}
As it will be emphasized in what follows, the regularity assumption is actually a consequence of the reach constraint:  indeed, the class of sets satisfying a ``Reach'' constraint is closed in a sense to be specified later and all elements are of class $\mathscr{C}^{1,1}$.
 
Other constraints such as a bound on the two-dimensional Hausdorff measure $\mathscr{H}^2(S)$ of $S$ (in other words the perimeter of the stellarator in $\R^3$) will also be considered:
\begin{equation}\tag{$\mathscr{H}_{{\rm Perim},P_{\max}}$}\label{Hyp_perim}
\mathscr{H}^2(S)\leq P_{\max}.
\end{equation}
\end{itemize}

    To sum-up, let us introduce the admissible set of shapes we will deal with in what follows:
    $$
    \Oad=\{S=\partial V \subset D \mid \text{$P\subset V$\text{ and $S$ satisfies }\eqref{H1}, \ \eqref{Hyp_reach}, \ \eqref{Hyp_perim}}\}.
    $$
     Note that the reach condition has been imposed on the surface and not only on the volume.

The resulting shape optimization problem we will consider reads
\begin{equation}\label{SOP}\tag{$\mathscr{P}_{\textrm{shape}}$}
\boxed{\inf_{S\in \Oad}C(S).}
\end{equation}

In the two following sections, we investigate two important aspects of the shape optimization problem~\eqref{SOP}. The first one concerns the existence of optimal shapes and is investigated in Section~\ref{sec:analyseExist}. 
The second one is related to the derivation of first order optimality conditions, at the heart of the algorithms implemented in the last section of this article. To this aim, we apply in Section~\ref{sec:shapeD} the so-called Hadamard boundary variation method recalled in Section~\ref{sec:Had}. 

\section{Existence issues for Problem~\eqref{SOP}}\label{sec:analyseExist}

\subsection[Existence of an optimal current for a given shape]{Existence of an optimal current for a given shape (Solving of Problem~\eqref{defpb:PS})}\label{sec:optCurrent}
We first establish that the infimum defining \eqref{defpb:PS} is in fact a minimum. Moreover, the minimizer is unique. 
\begin{lemma}
    \label{lem:existence}
    Let $S\in \Oad$. The optimization problem~\eqref{defpb:PS} has a unique minimizer $j_S$.
Moreover, one has
    \begin{align}
        j_S &= (\lambda \operatorname{Id}+\operatorname{BS}_S^\dagger \operatorname{BS}_S)^{-1}\operatorname{BS}_S^\dagger B_T,\nonumber\\
        C(S)&= \lambda\Vert (\lambda \operatorname{Id}+\operatorname{BS}_S^\dagger \operatorname{BS}_S)^{-1}\operatorname{BS}_S^\dagger B_T\Vert^2_{\mathscr{F}_S}+ \Vert \operatorname{BS}_S (\lambda \operatorname{Id}+\operatorname{BS}_S^\dagger \operatorname{BS}_S)^{-1}\operatorname{BS}_S^\dagger B_T -B_T \Vert^2_{L^2(P,\R^3)}.\label{eq:C(S)}
    \end{align}
\end{lemma}
\begin{proof}
   First observe  that $\mathscr{F}_S^0$ is a Hilbert space and that the mapping  $\mathscr{F}_S^0 \ni j\mapsto \|\operatorname{BS}_S j -B_T \|^2_{L^2(P,\R^3)} + \lambda \|j\|^2_{\mathscr{F}_S}$ is strongly convex, since it is the sum of the convex functional $
    j\mapsto \|\operatorname{BS}_S j -B_T \|^2_{L^2(P,\R^3)}$ and the strongly convex one $j\mapsto
    \lambda \|j\|^2_{\mathscr{F}_S}$. Furthermore, we claim that the functional $\mathscr{F}_S^0 \ni j\mapsto \|\operatorname{BS}_S j -B \|^2_{L^2(P,\R^3)} + \lambda \|j\|^2_{\mathscr{F}_S}$ is lower semicontinuous for the strong topology of $\mathscr{F}^0_S$. Indeed, let $(j_n)_{n\in \N}$ denote a sequence of $\mathscr{F}_S^0$ converging to $j\in \mathscr{F}_S^0$. According to \eqref{H1}, by using the dominated convergence theorem and since $K$ is uniformly bounded in $S\times P$, one has
    $$
    \lim_{n\to +\infty} \int_S K(x,y)\times j_n(x) d\mu_S(x)=\int_S K(x,y)\times j(x) d\mu_S(x).
    $$
    It follows that the functional to minimize is lower semicontinuous (and even continuous) in $\mathscr{F}_S^0$, whence the existence of a unique minimizer $j_S$ for Problem~\eqref{defpb:PS}.

    Since $\operatorname{BS}_S$ is continuous, its adjoint $\operatorname{BS}_S^\dag$ is well defined on $\mathscr{F}_S^0$. It is hence standard that the first order optimality condition for this problem reads
        \begin{equation}\label{cionOptimjS}
        \forall v\in \mathscr{F}_S^0,\qquad \langle \operatorname{BS}_S v , \operatorname{BS}_S j_S-B_T \rangle_{L^2(P,\R^3)}+\lambda \langle v , j_S \rangle_{\mathscr{F}_S^0} = 0
        \end{equation}
    which also rewrites 
    $$
    \forall v\in \mathscr{F}_S^0,\qquad \langle v , \lambda j_S+ \operatorname{BS}_S^\dagger(\operatorname{BS}_S j_S-B_T) \rangle_{\mathscr{F}_S^0} = 0.
    $$
    Since $v$ is arbitrary in $\mathscr{F}_S^0$, we thus infer that $ \lambda j_S+ \operatorname{BS}_S^\dagger \operatorname{BS}_S j_S =\operatorname{BS}_S^\dagger B_T$.
    The operator $\operatorname{BS}_S^\dagger \operatorname{BS}_S$ is compact and symmetric. Besides its spectrum is positive and we can therefore consider its resolvent for negative real numbers $-\lambda$ with $\lambda >0$, so that
$$j_S=  (\lambda \operatorname{Id}+\operatorname{BS}_S^\dagger \operatorname{BS}_S)^{-1} \operatorname{BS}_S^\dagger B_T.$$
The  expression of $C(S)$ given in \eqref{eq:C(S)} follows  from a straightforward computation.
\end{proof}

\begin{remark}
    \label{rmk:LS_simplifications}
    When confronted with the numerical implementation of 
   the shape optimization, 
motivated by     the structure of  $\mathscr{F}_S^0$ and the properties of the {\it in vacuo} Maxwell-equations,  
    we will find
    it useful to:
    \begin{itemize}
        \item optimize on a closed affine subset $j^a_S+\hat{\mathscr{F}_S^0}\subset \mathscr{F}_S^0$ instead of the entire set $\mathscr{F}_S^0$. We refer to Section \ref{subsec:current-sheet_representation} for further details;
        \item replace the target 
        magnetic field in $L^2(P,\R^3)$ by its normal component 
        on the plasma surface (thus, by an object  in $L^2(\partial P, \R)$).
        Indeed, a divergence-free vector field on a 3D domain (in absence of electric currents in the plasma) is nearly entirely characterized by its normal component on the boundary. 
        Further details are given in Section~\ref{subsec:Magnetic_field_representation} and Appendix~\ref{subsec:poisson_torus}.
    \end{itemize}
    Nevertheless, such changes have a minor impact on the theoretical discussion on the shape optimization process and we believe that, for the sake of clarity, it is better to postpone the details 
    about such 
 modifications to Section~\ref{sec:num}.
\end{remark}

\subsection{Existence of an optimal shape}\label{sec:existence_SOP}\label{sec:existSOP}

\begin{theorem}
    \label{th:existence_SOP}
    The shape optimization problem \eqref{SOP} has at least one solution.
\end{theorem}


The proof follows the direct method of the calculus of variation.
Most of the compactness on our set of admissible shapes comes from the bounded reach assumption \eqref{Hyp_reach}.
In particular, the following Lipschitz estimate is crucial.
\begin{lemma}[Theorem 2.8 of \cite{dalphin_uniform_2018}]
    \label{lem:Lipschitz}
    Let $V\subset{\R^n}$ be a nonempty set such that $\operatorname{Reach}(\partial V)\geq \rmin$ and $\mathscr{H}^{n}(\partial V)=0$. Then, for every $h\in (0,\rmin)$, the gradient 
    $\nabla b_V$
    of the signed distance function  is $\frac{2}{\rmin-h}$-Lipschitz on the tubular neighborhood $U_h(\partial V)$.
    
\end{lemma}
With this estimate, we can state the following compactness result. 
\begin{lemma} \label{lem:shape_compactness}
    Let $r$ be in $(0,r_{\min})$ and denote by $(S_n)_{n\in \N}=(\partial V_n)_{n\in \N}$ a sequence in $\Oad$.
Then, there exists $S_\infty=\partial V_\infty \in \Oad$ such that, up to a subsequence,
    \begin{itemize}
        \item $b_{V_\infty}$ is in $\mathscr{C}^{1,1}(\overline{U_{r}(S_\infty)})$ and $(b_{V_n})_{n\in \N}$ converges to $b_{V_\infty}$ in $\mathscr{C}^{1}(\overline{U_{r}(S_\infty)})$;
        \item $(b_{V_n})_{n\in\N}$ converges to $b_{V_\infty}$ in $\mathscr{C}(\overline{D})$;
        \item $(d_{S_n})_{n\in\N}$ converges to $d_{S_\infty}$ in $\mathscr{C}(\overline{D})$;
        \item $(\mathscr{H}^{2}(S_n))_{n\in \N}$ converges to $\mathscr{H}^{2}(S_\infty)$.
    \end{itemize}
\end{lemma}
\begin{proof}
Compactness properties among  Hausdorff 
  distances 
from sets of uniformly positive reach  are well known and remain valid for 
the signed distance (see, e.g., \cite[Chapter~6]{delfour_shapes_2011}).
    Besides, as stated in \cite{dalphin_uniform_2018}, the convergence property holds true for the strong topology of $\mathscr{C}^{1,\alpha}$ (for $\alpha<1$) and for the weak topology of $W^{2,\infty}$ in a tubular neighborhood of $S_\infty$.
    As a consequence,  $d(S_\infty,P)\geq \delta$ and $\operatorname{Reach}(\partial V_\infty)\geq \rmin$.
    In particular, thanks to Lemma~\ref{lem:Lipschitz}, $b_{V_\infty}$ is $\mathscr{C}^{1,1}$ on $\overline{U_{r}(S_\infty)}$.
Finally,  the convergence of $\mathscr{H}^{2}(S_n)$ to $\mathscr{H}^{2}(S_\infty)$ follows from standard results on   the continuity of $S\mapsto \mathscr{H}^2(S)$ (see \cite{dalphin_uniform_2018} or \cite{guo_convergence_2013}).
%
\end{proof}

The end of this section is devoted to the proof of Theorem~\ref{th:existence_SOP}.
    Let $(S_n)_{n\in \N}=(\partial V_n)_{n\in\N}$ be a minimizing sequence for Problem~\eqref{SOP}. Denote by $S_\infty$ a closure point of this sequence in the sense of Lemma~\ref{lem:shape_compactness}.
 In what follows, we will still denote by $(S_n)_{n\in \N}$ the converging subsequence introduced in Lemma~\ref{lem:shape_compactness}.
 
 We will proceed by showing a semicontinuity property of the criterion, namely that 
 $$
 \liminf_{n\to +\infty} C(S_n)\geq C(S_\infty).
 $$

Let $j_n$ denote the minimizer of Problem~\eqref{defpb:PS} for the surface $S_n$, whose existence is provided by Lemma~\ref{lem:existence}.

The idea is to consider a volume integral as approximation of the surface integral in the same spirit as in \cite{delfour_tangential_2000}. For this purpose, we need to extend locally $j_n$ to a volume around $S_n$. 
%
Notice that, 
without loss of generality, $S_n$ is contained in $U_{\rmin}(S_\infty)$ for every $n$, which implies, in particular, that $\nabla b_{V_{\infty}}$ is everywhere defined and Lipschitz continuous on $S_n$. 
Let $h>0$ be a small constant to be fixed later and 
define the map 
\begin{equation*}
T_n:
\begin{array}[t]{rcl}       
(-h,h) \times S_n &\to &A_h(S_n) \subset U_h(S_n) \\
    (t,x)  &\mapsto &x+t \nabla b_{V_{\infty}}(x),
\end{array}
\end{equation*}
where $A_h(S_n)$ denote the image of $T_n$.
Notice that $T_n$ is a bijection between $(-h,h) \times S_n$ and $A_h(S_n)$ if the latter is contained in $U_{\rmin}(S_\infty)$ (see Figure~\ref{fig:reach}).

    The differential of $T_n$ at $(t_0,x_0)\in(- h, h)\times S_n$ reads
$$
DT_n(t_0,x_0): \begin{array}[t]{rcl}
  \R\times T_{x_0}S_n &\to &\R^3\\
            (s,y) &\mapsto & s \nabla b_{V_\infty}(x_0)+ y+ t_0 \nabla_{S_n} (\nabla b_{V_\infty})(x_0) y,
    \end{array}
$$
where $\nabla_{S_n} (\nabla b_{V_\infty})(x_0)$ is a $3\times 3$ matrix, according to the notation introduced in Section~\ref{sec:notations}, and $T_{x_0}S_n$ is identified with a linear subspace of $\R^3$. 
We can identify $DT_n(t_0,x_0)$ with a $3\times 3$ matrix by choosing an orthogonal basis on $T_{x_0}S_n$ and its determinant, denoted $|DT_n(t_0,x_0)|$ in what follows, is independent of such a choice. 

\begin{figure}[!htb]
    \centering
    \begin{minipage}{.45\textwidth}
      \centering
      \includegraphics[height=1\linewidth]{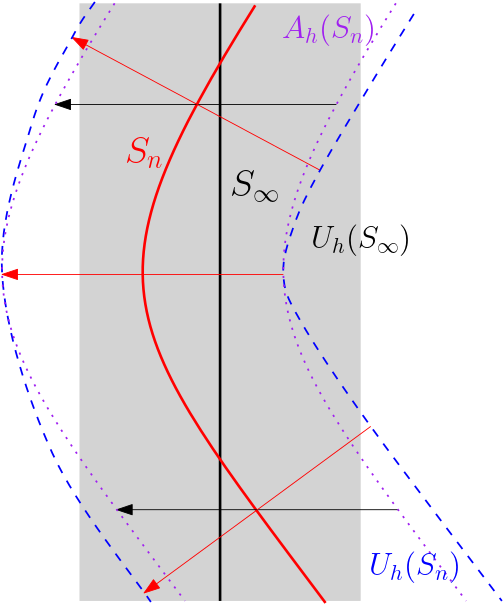}
      \captionof{figure}{This figure illustrates the difference between $U_h(S_\infty)$ filled in grey, $U_h(S_n)$ (resp., $A_h(S_n)$) delimitated by the blue dashed (resp., purple dotted) curves. The black arrows represent the field $\nabla b_{V_\infty}$ and the red ones represent $\nabla b_{V_n}$. Note that both $V_n$ and $V_\infty$ are on the right of the figure.}
      \label{fig:reach}
    \end{minipage}%
    \hfill
    \begin{minipage}{.45\textwidth}
      \centering
      \includegraphics[height=1\linewidth]{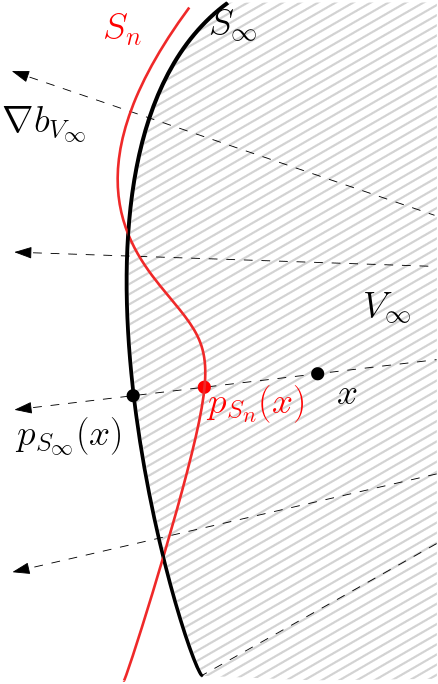}
      \captionof{figure}{$p_{S_n}(x)$ is obtained by taking the intersection of the flow of $\nabla b_{V_\infty}$ and $S_n$. Whereas the standard projector (in the sense of shortest distance) on $S_n$ is obtained by using the flow of $\nabla b_{V_n}$.}
      \label{fig:proj}
    \end{minipage}
    \end{figure}

Using the regularity of $\nabla b_{V_\infty}$ near the surface $S_\infty$, one can prove the following crucial estimate.
\begin{lemma}\label{lem:det_bounded}
    For every $\eps>0$, there exists $h=h_\eps>0 
    $ such that 
\begin{align*}
1-\eps \leq |DT_n(y)|\leq 1+\eps,\quad \text{for a.e. }(t_0,x_0)\in 
(-h,h)
\times S_n
\end{align*}
for every $n \in \N \cup \{ \infty \}$ large enough.
\end{lemma}
\begin{proof}
By Lemma~\ref{lem:shape_compactness},     $b_{V_\infty}$ is $\mathscr{C}^{1,1}$ in a neighborhood of $S_\infty$, thus
 $\nabla b_{V_\infty}|_{S_n}$ is Lipschitz continuous for $n$ large enough with a Lipschitz constant independent of $n$ and, in particular, there exists $C>0$ such that for all $n$ large enough
    $$
    \left|\nabla_{S_n}(\nabla b_{V_\infty})\right|_{L^\infty(S_n)}\leq C.
    $$
    Besides, the linear mapping $R_{x_0}:(-h,h)\times T_{x_0}S_n \ni (s,y)\mapsto s \nabla b_{V_\infty}(x_0)+ y$ is direct and orthogonal since $\nabla b_{V_\infty}(x_0)$ is the unit normal outward vector. Hence, its determinant is equal to 1. Since 
     the determinant is $\mathscr{C}^\infty$
    , we have
    $$
    \sup_{(t_0,x_0)\in (-h,h)\times S_n} |  |DT_n(t_0,x_0)| -\det R_{x_0} | =\operatorname{O}(h),
  $$
where the reminder term is uniformly bounded with respect to $n$ for $n$ large enough. This concludes the proof.
\end{proof}


In what follows $\eps>0$ is a small parameter to be fixed and $h$ is as in the statement of Lemma~\ref{lem:det_bounded}. We shall also assume that $A_h(S_n)\subset U_{\rmin}(S_\infty)$ for every $n$. 
Notice that, as soon as $\eps<1$, for $n$ large enough  $T_n$ is a diffeomorphism. Thus we can define on $A_h(S_n)$ the projector $p_{S_n}$ onto $S_n$ along the 
field $\nabla b_{V_{\infty}}$ by requiring that $p_{S_n}$ coincides with the $S_n$-component  of the inverse of $T_n$ (see Figure~\ref{fig:proj}).

This allows us to introduce $\tilde{\jmath}_n$, defined on $A_h(S_n)$ by
 $$
         \tilde{\jmath}_n=j_n\circ p_{S_n}.
  $$

Using \cite[Chap. 7, theorem 8.5]{delfour_shapes_2011}, for any $m \in \N \cup \{\infty\}$ large enough, we get
\begin{align}
    \label{eq:permutation_p}
    p_{S_n}=p_{S_n} \circ p_{S_m}
\end{align}
on $ A_{h}(S_n)\cap A_{h}(S_m)$.
Moreover, $p_{S_n}$ converges uniformly to $p_{S_\infty}$ in a neighborhood of $S_\infty$, since 
for every $x\in U_{\rmin}(S_\infty)$ one has
\[| p_{S_n}(x)-p_{S_\infty}(x)|=d_{S_\infty}(p_{S_n}(x))\leq \|d_{S_\infty}-d_{S_n}\|_{L^\infty(\overline{D})}\to 0\qquad \mbox{as }n\to\infty,\]
where the limit is a consequence of Lemma~\ref{lem:shape_compactness}.

Using the change of variable formula (also known as area formula for Lipschitz functions), one gets for 
$n\in\N\cup\{\infty\}$ large enough, every $f\in L^1(S_n)$, and every $\kappa\in (0,h)$,
    \begin{align*}
        \int_{-\kappa}^\kappa \int_{S_n} f(x)\, d\mu_{S_n}(x) dt= \int_{A_\kappa(S_n)}f\circ p_{S_n}(y) |DT_n(T_n^{-1}(y))| \, dy,
    \end{align*}
    which also rewrites
    \begin{align}
        \label{eq:integral_volume_surface_bis}
     \int_{S_n} f(x)\, d\mu_{S_n}(x)= \frac{1}{2\kappa}\int_{A_\kappa(S_n)} f\circ p_{S_n}(y) |DT_n(T_n^{-1}(y))| \, dy.
    \end{align}

     Let $\kappa$ be in $(0,h)$ and $\eta>0$ be small enough so that $\kappa-\eta>0$ and $\kappa+\eta<h$, that is $\eta<\min(\kappa,h-\kappa)$. Since $d_{S_n}\to d_{S_\infty}$ in $\mathscr{C}(\overline{D})$, there exists $N$ such that for all $n>N$,
    \begin{align}
        \label{eq:tubular_inclusion}
        A_{\kappa-\eta}(S_n) \subset A_{\kappa}(S_\infty) \subset A_{\kappa+\eta}(S_n).
    \end{align}
 Using that with Equation \eqref{eq:permutation_p}, we obtain
    \begin{align}
        \int_{S_\infty} |\tilde{\jmath}_n(x)|^2d\mu_{S_\infty}(x)&= \frac{1}{2 \kappa} \int_{A_{\kappa}(S_\infty)} |\tilde{\jmath}_n(y)|^2 |DT_\infty(T_\infty^{-1}(y))| \, dy\nonumber \\
        &\leq \frac{1}{2 \kappa } \int_{A_{\kappa+\eta}(S_n)} |\tilde{\jmath}_n(y)|^2 |DT_\infty(T_\infty^{-1}(y))| dy\nonumber\\
        &= \frac{1}{2 \kappa } \int_{A_{\kappa+\eta}(S_n)} |\tilde{\jmath}_n(y)|^2 \frac{|DT_\infty(T_\infty^{-1}(y))|}{|DT_n(T_n^{-1}(y))|}|DT_n(T_n^{-1}(y))| dy\nonumber\\
        &\leq \frac{\kappa +\eta}{\kappa}\frac{1+\eps}{1-\eps} \|j_n\|^2_{\F_{S_n}},\label{formula:1147}
    \end{align}
    which ensure that $\tilde{\jmath}_n$ belongs to $L^2(S_\infty, \R^3)$. (Notice however that $\tilde{\jmath}_n$  is not necessarily in $\mathscr{F}_{S_\infty}^0$, as it is neither, in general, a tangent vector field nor a divergence free one.)
Equation \eqref{formula:1147} actually shows that  $(\tilde{\jmath}_n)_{n\in \N}$ is bounded in $L^2(S_\infty,\R^3)$. Up to subsequence, it converges weakly to $j_\infty \in L^2(S_\infty,\R^3)$ with 
    $$
    \|j_\infty \|^2_{L^2(S_\infty,\R^3)}\leq \liminf_{n\to +\infty} \|j_n \|^2_{\mathscr{F}_{S_n}}.
    $$

    The remaining two steps of the proof consist first in showing the semicontinuity property
    \begin{equation}\label{eq:semicontinu}
    \|\operatorname{BS}_{S_\infty} j_\infty -B_T \|^2
    _{L^2(P,\R^3)} \leq 
    \liminf_{n\to +\infty} \|\operatorname{BS}_{S_n} j_n -B_T \|^2
    _{L^2(P,\R^3)}
    \end{equation}
    and then in checking that $j_\infty$ belongs to $\mathscr{F}_{S_\infty}^0$.
    Notice that, even if we have defined the operator $\operatorname{BS}_{S_\infty}$ only among the vector fields tangent to $S_\infty$, by a slight abuse of notation it still makes sense to consider $\operatorname{BS}_{S_\infty} j_\infty$, defined using formula \eqref{eq:BS-with-kernel}, even without having checked that $j_\infty$ is in $\mathscr{F}_{S_\infty}$.
    
  Both steps rely on the following lemma.
  
    \begin{lemma}\label{lem:step-co}
     Given $C>0$ and $\eps'>0$, there exists $N\in \N$ 
 such that for every 
 $r>0$ and every  $f\in \mathscr{C}^1(U_r(S_\infty))$  such that
$\|f\|_{L^\infty(U_r(S_\infty))}\le C$ and $f$ is $C$-Lipschitz continuous on 
$U_r(S_\infty)$, we have 
\[\left|\int_{S_n}f(x)j_n(x)d\mu_{S_n}(x)-\int_{S_\infty}f(x)\tilde{\jmath}_n(x)d\mu_{S_\infty}(x)\right|<\eps'\]
 for $n> N$.
\end{lemma}
    \begin{proof}
    Let $r>0$ and  $f\in \mathscr{C}^1(U_r(S_\infty))$ be such that
$\|f\|_{L^\infty(U_r(S_\infty))}\le C$ and $f$ is $C$-Lipschitz continuous on 
$U_r(S_\infty)$.
  Up to taking $h$ small enough, we can assume that 
  \begin{equation}\label{eq:udl}
  A_h(S_n)\subset U_r(S_\infty)\qquad \mbox{for $n$ large enough}. 
    \end{equation}
    
As above, consider $\kappa\in (0,h)$, $0<\eta<\min(\kappa,h-\kappa)$, and $n$ large enough so that \eqref{eq:tubular_inclusion} holds true. 
By \eqref{eq:integral_volume_surface_bis}, we have
    \begin{align*}
 \left|\int_{S_n}f(x)j_n(x)d\mu_{S_n}(x)-\int_{S_\infty}f(x)\tilde{\jmath}_n(x)d\mu_{S_\infty}(x)\right|
        ={}&   \bigg| \frac{1}{2(\kappa-\eta)} \int_{A_{\kappa-\eta}(S_n)} f(p_{S_n}(x)) \tilde{\jmath}_n(x) |DT_n| \, dx\\
      &  - \frac{1}{2\kappa} \int_{A_{\kappa}(S_\infty)} f(p_{S_\infty}(x))\tilde{\jmath}_n(x) |DT_\infty| \, dx \bigg|,
    \end{align*}
     where, for notational simplicity, we write   $|DT_n|$ for $|DT_n(T_n^{-1}(x))|$
   and $|DT_\infty|$ for $|DT_\infty(T_\infty^{-1}(x))|$. 
   Hence, 
   \[    \left|\int_{S_n}f(x)j_n(x)d\mu_{S_n}(x)-\int_{S_\infty}f(x)\tilde{\jmath}_n(x)d\mu_{S_\infty}(x)\right|\le 
 A_1+A_2\]
 where we added and subtracted $\frac{1}{2\kappa}\int_{A_{\kappa-\eta}(S_n)} f(p_{S_\infty}(x))\tilde{\jmath}_n(x) |DT_\infty| \, dx$ to get
 \begin{align*}      
 A_1&= \left| \int_{A_{\kappa-\eta}(S_n)}\left( \frac{1}{2(\kappa-\eta)} f(p_{S_n}(x)) \tilde{\jmath}_n(x) |DT_n| - \frac{1}{2\kappa} f(p_{S_\infty}(x)) \tilde{\jmath}_n(x) |DT_\infty| \right) dx \right|, \\
A_2 &= \left| \frac{1}{2\kappa}\int_{A_{\kappa}(S_\infty)\setminus A_{\kappa-\eta}(S_n)} f(p_{S_\infty}(x))  \tilde{\jmath}_n(x) |DT_\infty| \, dx \right|.
    \end{align*}
    
    We are going to show that $A_1$ and $A_2$ can be made arbitrarily small by suitably choosing  $\kappa$ and $\eta$ (depending only on $C$ and not on the specific function $f$) and letting $n$ be large enough.  
    
    The  term $A_2$ can be estimated using the inequality $\|f\|_{L^\infty(U_r(S_\infty))}\le C$, as follows: 
    \begin{align*}
        A_2 &\leq C\frac{1+\eps}{1-\eps}\int_{A_{\kappa}(S_\infty)\setminus A_{\kappa-\eta}(S_n)} |\tilde{\jmath}_n(x) | |DT_n|\, dx 
        \leq C \frac{1+\eps}{1-\eps}\int_{A_{\kappa+\eta}(S_n)\setminus A_{\kappa-\eta}(S_n)} |\tilde{\jmath}_n(x) | |DT_n|\, dx\\
        & = 4 \eta C\frac{1+\eps}{1-\eps}\|j_n\|_{L^1(S^n)}, 
         \end{align*}
         where the factor $4$ comes from the fact that the measure of $(-\kappa-\eta,-\kappa+\eta)\cup (\kappa-\eta,\kappa+\eta)$ is equal to $4\eta$. 
   Notice that     $\|j_n\|_{L^1(S^n)}  \le \|j_n\|_{\mathscr{F}_{S_n}} \sqrt{\mathscr{H}^2(S_n)}$ is bounded uniformly with respect to $n$, so that $A_2$ can be made arbitrarily small by choosing $\eta$ small enough (depending only on $C$). 
         
        Let us 
       now focus on the term $A_1$. 
Since
        $$|f(x_1)-f(x_2)|\leq C |x_1-x_2|,\qquad \forall x_1,x_2\in U_{r}(S_{\infty}),$$
and because of \eqref{eq:udl},        it follows that
        $$
        \sup_{x\in A_{\kappa-\eta}(S_n)} |f(p_{S_n}(x))-f(p_{S_\infty}(x))|\leq C \|p_{S_n}-p_{S_\infty}\|_{L^\infty(A_{\kappa-\eta}(S_n))}\leq C \|p_{S_n}-p_{S_\infty}\|_{L^\infty(A_{h}(S_\infty))} $$
    for $n$ large enough.

    Hence, we have the estimates
    \begin{align*}
    A_1\leq{}& \frac{1}{2(\kappa-\eta)}\bigg| \int_{A_{\kappa-\eta}(S_n)}\left(  f(p_{S_n}(x)) \tilde{\jmath}_n(x) |DT_n| -  f(p_{S_\infty}(x)) \tilde{\jmath}_n(x) |DT_n| \right) dx \bigg| \\
    &+ \bigg| \int_{A_{\kappa-\eta}(S_n)} \left(\frac{1}{2(\kappa-\eta)} f(p_{S_\infty}(x)) \tilde{\jmath}_n(x) |DT_n| - \frac{1}{2\kappa} f(p_{S_\infty}(x))\tilde{\jmath}_n(x) |DT_n| \right) dx \bigg| \\
    &+ \frac{1}{2\kappa}\bigg| \int_{A_{\kappa-\eta}(S_n)} \left( f(p_{S_\infty}(x))  \tilde{\jmath}_n(x) |DT_n| -  f(p_{S_\infty}(x))  \tilde{\jmath}_n(x) |DT_\infty| \right) dx \bigg| \\
    \leq{}&  C \|p_{S_n}-p_{S_\infty}\|_{L^\infty(A_h
    (S_\infty))} \|j_n(x)\|_{L^1(S_n)} + C \|j_n(x)\|_{L^1(S_n)} \left(1-\frac{\kappa-\eta}{\kappa}\right)\\
    &+ C \frac{\kappa-\eta}{\kappa} \|j_n(x)\|_{L^1(S_n)} \left(1- \frac{1+\eps}{1-\eps}\right).
    \end{align*}
    Hence $A_1$ can be made arbitrarily small choosing $\eps$ and then $\eta$ small enough, and letting $n$ large enough. 
       \end{proof}

      Let us start the proof of \eqref{eq:semicontinu} by  comparing $\operatorname{BS}_{S_n}$ and $\operatorname{BS}_{S_\infty}$. 
      Given $y\in P$, one has
    \begin{align*}
 | \operatorname{BS}_{S_n}(j_n)(y)-\operatorname{BS}_{S_\infty}(\tilde{\jmath}_n)(y) |
        &=   \bigg| \int_{S_n} K(x,y) \times j_n(x)\, dx
        - \int_{S_\infty} K(x,y) \times \tilde{\jmath}_n(x) \, dx \bigg|.
    \end{align*}
    
  Notice that $|K(\cdot,y)|$ is bounded in a neighborhood of $S_\infty$, 
  uniformly  with respect to $y\in P$, since $\sup_{(x,y)\in S_n\times P} |K(x,y)| \leq \frac{1}{\delta^2}$.
  Moreover, for every $y\in P$ and every $\rho>0$, the map 
       $x \mapsto \|D_x K(x,y)\|$ is upper bounded by $\frac{4}{\rho^3}$ outside $U_\rho(P)$ according to \eqref{eq:der_K}. 
Assume that $h<\delta$, so that     $A_{h}(S_n)$ is at distance at least $\delta-h$ from $P$ for every $n$.        
           Consider $\rho<\delta-h$ and 
       a  Lipschitz neighborhood ${\cal N}$ of  $\R^3 \setminus U_{\delta-h}(P)$
       not intersecting $U_\rho(P)$. Since the geodesic distance in ${\cal N}$  is equivalent to the restriction to ${\cal N}$ of the standard Euclidean distance,  
       we deduce that there exists $\tilde C>0$ independent of $y$ such that
        $$|K(x_1,y)-K(x_2,y)|\leq \tilde C |x_1-x_2|,\qquad \forall x_1,x_2\in \R^3\setminus U_{\delta-h}(P).$$
We deduce from Lemma~\ref{lem:step-co} that 
for every $\eps'>0$ there exists $N>0$ such that for any integer $n>N$,
    $$\|\operatorname{BS}_{S_n}(j_n)-\operatorname{BS}_{S_\infty}(\tilde{\jmath}_n)\|_{L^2(P,\R^3)}\leq \eps',$$
    and, in particular,
    $$
    \|\operatorname{BS}_{S_\infty} \tilde{\jmath}_n -B_T \|
    _{L^2(P,\R^3)} \leq 
    \|\operatorname{BS}_{S_n} j_n -B_T \|
    _{L^2(P,\R^3)}+ \eps'. 
    $$
    Using the compactness of $\operatorname{BS}_{S_\infty}$, we have 
    $$
    \|\operatorname{BS}_{S_\infty} j_\infty -B_T \|
    _{L^2(P,\R^3)} \leq 
    \liminf_{n\to +\infty} \|\operatorname{BS}_{S_n} j_n -B_T \|
    _{L^2(P,\R^3)}+ \eps'. 
    $$
    This concludes the proof of \eqref{eq:semicontinu}, since $\eps'$ is arbitrary. 
    
    To conclude the proof, it remains to check that $j_\infty$ belongs to $\mathscr{F}_{S_\infty}^0$.
    By weak convergence of $\tilde{\jmath}_n$ to $j_\infty$ and according to Lemma~\ref{lem:step-co},
       \begin{align*}
     \|\langle j_\infty, \nabla b_{V_\infty}\rangle\|_{L^2(S_\infty ,\R^3 )}&=
             \lim_{n\to \infty}\|\langle\tilde{\jmath}_n, \nabla b_{V_\infty}\rangle\|_{L^2(S_\infty ,\R^3 )}
             =\lim_{n\to \infty}\|\langle j_n, \nabla b_{V_\infty}\rangle\|_{L^2(S_n ,\R^3 )}\\
             &\le \limsup_{n\to \infty}\|\langle j_n, \nabla b_{V_\infty}-\nabla b_{V_n}\rangle\|_{L^2(S_n ,\R^3 )},
    \end{align*}
    where we used that $ j_n$ is  orthogonal to $\nabla b_{V_n}$ everywhere on $S_n$. 
  According to Lemma~\ref{lem:Lipschitz}, moreover, 
    \begin{align*}
\lim_{n\to \infty}            \|\nabla b_{V_\infty} -\nabla b_{V_n}\|_{L^\infty(S_\infty, \R^3)}=0,
    \end{align*}
 and we conclude that $\|\langle j_\infty, \nabla b_{V_\infty}\rangle\|_{L^2(S_\infty ,\R^3 )}=0$ 
 since the sequence $(\| j_n\|_{L^2(S_\infty,\R^3)})_{n\in\N}$ is bounded.
  This proves that $j_\infty$ is a vector field tangent to $S_\infty$.

    To prove that $j_\infty$ is divergence free (in distributional sense), we have to check that $ j_\infty$ is orthogonal to $\{\nabla_{S_\infty} f \mid f\in\mathscr{C}^1(S_\infty)\}$. Indeed, this 
  characterization of divergence-free vector fields   
    follows from the Hodge decomposition (see Appendix~\ref{append:Hodge}).
    For $g\in \mathscr{C}^\infty(\R^3)$, 
    since $\dive j_n=0$ on $S_n$, one has
\begin{equation*}
  0=  \int_{S_n} \langle j_n, \nabla_{S_n}g\rangle d\mu_{S_n}=  \int_{S_n} \langle j_n,  \left( \nabla g -\langle \nabla g, \nabla b_{V_n}\rangle \nabla b_{V_n} \right) \rangle d\mu_{S_n}.
\end{equation*}

Set $G_n:=\nabla g -\langle \nabla g, \nabla b_{V_n}\rangle \nabla b_{V_n}$ for $n\in \N\cup \{\infty\}$. 
Notice that $G_n$ converges uniformly to $G_\infty$ in a neighborhood of $S_\infty$. 
Hence, again using Lemma~\ref{lem:step-co},
       \begin{align*}
  \int_{S_\infty} \langle j_\infty, G_\infty\rangle d\mu_{S_\infty}(x) &=
             \lim_{n\to \infty}\int_{S_\infty} \langle \tilde{\jmath}_n, G_\infty\rangle d\mu_{S_\infty}(x)=
             \lim_{n\to \infty}\int_{S_n}\langle  \tilde{\jmath}_n, G_\infty\rangle d\mu_{S_n}(x)\\
              &=
             \lim_{n\to \infty}\int_{S_n}\langle  \tilde{\jmath}_n, G_n\rangle d\mu_{S_n}(x)=0.
    \end{align*}
  
    This concludes the proof of Theorem \ref{th:existence_SOP}.
\section{Shape differentiation for Problem~\eqref{SOP}}\label{sec:diffSOP}
\subsection{Reminders on the Hadamard boundary variation method}\label{sec:Had}

Let us recall hereafter some notions of topology on sets of regular domains 
defined in terms of 
particular perturbations called \emph{identity perturbations}. 
The latter are of the form $\tau=\operatorname{Id}+\theta$, 
where $\theta$ is small enough in a suitable sense. More precisely, according to the approach developed in \cite{MuratSimon1976,Simon}, one defines
\begin{equation*}
\mathcal{V}^{k,\infty}=\{\tau:\R^3\rightarrow \R^3\mid (\tau-\operatorname{Id})\in W^{k,\infty}(\R^3,\R^3) \},
\end{equation*}
and
\begin{equation*}
\mathcal{T}^{k,\infty}=\{\tau:\R^3\rightarrow \R^3\mid \tau\in \mathcal{V}^{k,\infty}, \ \tau\textrm{ is invertible,  and }\tau^{-1}\in \mathcal{V}^{k,\infty} \},
\end{equation*}
with $k\in \N^*$. 

Let us recall that the space $W^{k,\infty}(\R^3,\R^3)$ endowed with the norm
$$
\Vert \tau\Vert_{W^{k,\infty}}=\underset{0\leq |\alpha|\leq k}{\operatorname{sup}} \Vert D^\alpha u\Vert_\infty
$$
is a Banach space. 
Choosing $\tau$ in $\mathcal{T}^{2,\infty}$ 
allows to
preserve the topological and regularity properties of sets we are interested in, as highlighted in the next result. 

\begin{lemma}\label{lem:HadamProp}
Let $\Omega_0$ be an open bounded subset of $\R^3$ and let $\tau\in \mathcal{T}^{2,\infty}$. 
\begin{itemize}
\item If $\partial\Omega_0$ is of class $\mathscr{C}^{1,1}$, then  $\tau(\Omega_0)$ is an open bounded domain whose boundary is of class $\mathscr{C}^{1,1}$. Furthermore, one has $\tau(\partial\Omega_0)=\partial(\tau(\Omega_0))$.
\item If $\theta\in W^{2,\infty}(\R^3,\R^3)$ is such that $\Vert \theta\Vert_{W^{2,\infty}}<1$, then $\operatorname{Id}+\theta\in \mathcal{T}^{2,\infty}$.
\end{itemize}
\end{lemma}

The first statement of this lemma follows from standard arguments. For instance, the property ``$\tau(\partial\Omega_0)=\partial(\tau(\Omega_0))$'' directly results from the fact that $\tau$ defines a homeomorphism. The second one comes from a direct application of the Banach fixed-point theorem. We refer to \cite[Chapter~4]{delfour_shapes_2011} for more explanations.

As a consequence of the lemma, 
$\mathcal{T}^{2,\infty}$ 
induces
a topology on the set $\mathcal{O}^1$ of open sets of $\R^3$ whose boundary belongs to $\mathscr{C}^{1,1}$  (according for instance to \cite[Assertion 2.52]{MuratSimon1976}). Given $\Omega_0\in\mathcal{O}^1$, 
let $\mathscr{V}(\Omega_0,\varepsilon)$ be the set of domains of the type $(\operatorname{Id}+\theta)\Omega_0$ with $\Vert \theta\Vert_{W^{2,\infty}}\leq \eps$ and $\varepsilon>0$ small enough so that $\operatorname{Id}+\theta$ is a diffeomorphism (Lemma~\ref{lem:HadamProp}). Then, one defines a topology on the space $\mathcal{O}^1$ with the help of 
the neighborhood basis given by the sets of $\mathscr{V}(\Omega_0,\varepsilon)$. Furthermore, it is shown in \cite{micheletti} that every neighborhood of $\Omega_0$ in $\mathscr{V}(\Omega_0,\varepsilon)$  is metrizable with a Courant-type  distance (induced by that 
 associated with $\|\cdot\|_{W^{2,\infty}}$ in ${\cal T}^{2,\infty}$) 
and 
has the structure of complete separable manifold.



Let us conclude this section by recalling the notion of shape differentiability.
\begin{definition}
A shape functional $\Omega\mapsto J(\Omega)$ is said to be \emph{shape differentiable at $\Omega$} (in the sense of Hadamard) in the class of domains with $\mathscr{C}^{1,1}$ boundary
whenever the underlying mapping 
$$
W^{2,\infty}(\R^{3},\R^{3})\ni \theta  \mapsto J(\Omega_{\theta })\in \R,
 $$ 
with $\Omega_{\theta } = (\mathrm{Id} + {\theta })(\Omega)$,
 is differentiable in the sense of Fr\'echet  at $\theta =0$. 
The corresponding differential $\langle dJ(\Omega),\cdot \rangle $ is the so-called \emph{shape derivative of $J$ at $\Omega$} and, by definition of Fr\'echet differential,  the following expansion holds:
\begin{equation*}
 J(\Omega_{\theta }) = J(\Omega) + \langle dJ(\Omega),\theta \rangle  + \operatorname{o}(\theta ), \qquad \mathrm{where~} \frac{\operatorname{o}(\theta )}{ \Vert \theta \Vert_{W^{2,\infty}(\R^3,\R^3) }} \xrightarrow[\theta  \to 0]{} 0.
\end{equation*}
\end{definition}
In the next section we study the shape differentiability of the cost $C$. In order to fit Definition 2, $C$ is implicitly identified with a functional $V\mapsto C(\partial V)=C(S)$ on the set of $\mathscr{C}^{1,1}$ toroidal domains. 

\subsection{Shape derivative of the cost functional $C$}\label{sec:shapeD}

This section and the next one are devoted to the computation of the shape derivative of the functional $C$.

\begin{theorem}\label{theo:shapeDer}
Let $S=\partial V\in \mathscr{O}_{ad}$. Let $Z_P\in \mathcal{L}(L^2(P,\R^3),\mathcal{F}_S)$ and $\widehat{Z}_P$, a bilinear mapping from $L^2(P,\R^3)\times \mathcal{F}_S^0$ into $\mathcal{F}_S$, defined by
\begin{eqnarray*}
Z_P(k) &=& \int_P K(\cdot,y)\times k(y) \, d\mu_P(y),\\
\widehat{Z}_P(k,j)(x) &=& \int_P
 D_x\left(\frac{x-y}{|x-y|^3}\right)^T \big(k(y) \times j(x) \big)d\mu_P(y), \qquad \forall x\in S.\\
\end{eqnarray*}

The functional $C$ defined by \eqref{defpb:PS} is shape differentiable at $S$. 
 Moreover, for every $\theta\in W^{2,\infty}(\R^3,\R^3)$ one has
\begin{equation*}
\langle dC(S),\theta\rangle =  \int_S \langle\theta ,( X_1  -\operatorname{div}_S(X_2)_{i:} )\rangle  \, d\mu_S
\end{equation*}
with
\begin{align*}
    X_1&= -2 \widehat{Z}_P(\operatorname{BS}_S j_S- B_T,j_S),\\
    X_2&= -2 Z_P (\operatorname{BS}_S j_S- B_T) j_S^T +2 \lambda j_S j_S^T - \lambda |j_S|^2 (I_3-\nu\nu^T),
\end{align*}
where for $i\in \{1,2,3\}$, $(X_2)_{i:}$ denotes the $i$-th line of $X_2$ seen as a column vector, and $\nu$ denotes the outward normal vector to $S=\partial V$.
\end{theorem}

\begin{remark}
The proof of this result relies crucially on the expression of the magnetic field provided through the Biot and Savart operator $\operatorname{BS}_S$ (see Definition~\ref{def:opBS}). In general, in many shape optimization problems involving PDEs on bounded domains, PDEs are interpreted as implicit equations on the deformation variable $\theta$ and on the state variable. They are in general taken into account by applying the implicit function theorem which also provides an expression for the material (or Lagrangian) derivative of the state with respect to the deformation (see, e.g., \cite[Chapter 5]{henrotShapeVariationOptimization2018}). In the present case, dealing with the Biot and Savart operator comes to consider a PDE on an unbounded domain. The approach we have chosen here, instead, is based on the integral representation of the state variable (the magnetic field here). To establish the above result, we use suitable changes of variables that allows us to rewrite the criterion as an integral over a fixed domain and derive it as a parameterized integral with respect to $\theta$. Although the principle of this calculation is simple, its implementation is not straightforward. 
\end{remark}

\subsection{Proof of Theorem~\ref{theo:shapeDer}}

For the sake of notational simplicity, the inverse of a group element $\varphi^\eps$ will be denoted with a slight abuse of notation by $\varphi^{-\eps} := (\varphi^\eps)^{-1}$.

Let $S$ and $\theta$ be as in the statement of the theorem. Assume for now that the criterion $C$ is shape differentiable at $S$. 
We will comment on this assumption at the end of the proof. In what follows, we concentrate on the computation of the shape derivative in the direction $\theta$.

Since $C$ is shape differentiable at $S$, 
we infer that
$$
\langle dC(S),\theta\rangle=\left.\frac{d}{d\eps}C(S^\eps)\right|_{\eps=0}, \qquad \text{with }S^\eps=(\operatorname{Id}+\eps \theta)S.
$$
\paragraph{Step 1: a change of variable.}

In order to compute $C(S^\eps)$, we need to compute some kind of derivative of $\operatorname{BS}_{S^\eps}$ and its adjoint.
Nevertheless, we aim to overcome the fact that the domain of $\operatorname{BS}_{S^\eps}$ 
depends on $\eps$.

Notice that, according to the discussion in Section~\ref{sec:Had}, the mapping $\varphi^\eps=\operatorname{Id}+\eps \theta$ induces a 
bijection 
between $\mathfrak{X}(S)$ 
and $\mathfrak{X}(S^\eps)$.
Nevertheless 
$\varphi^\eps$ does not map 
$\mathscr{F}_S^0$ into $\mathscr{F}_{S^\eps}^0$. This leads us to introduce the linear mapping
\begin{equation}
    \label{def:Phi^eps}
    \begin{split}
    \Phi^\eps : \mathscr{F}_S & \longrightarrow \mathscr{F}_{S^\eps}\\
    X &\longmapsto  \frac{1}{[J(\mu_S,\mu_S^\eps)\varphi^\eps]\circ \varphi^{-\eps}} (\operatorname{Id}+\eps D\theta)X\circ \varphi^{-\eps},
\end{split}    
\end{equation}
where $J(\mu_S,\mu_S^\eps)\varphi^\eps$ denotes the Jacobian determinant\footnote{Note that  $J(\mu_S,\mu_S^\eps)\varphi^\eps$ is not the determinant of the three-dimensional mapping $(\operatorname{Id}+\eps D\theta)$
but the determinant of the restriction of this application from $T_xS$ (the tangent space of $S$ at $x$) into $T_{(\operatorname{Id}+\eps \theta)(x)}S^\eps$.} of $\varphi^\eps$ (see Appendix~\ref{appendix:cov} for further details and the explicit expression of $J(\mu_S,\mu_S^\eps)\varphi^\eps$).

The following result will be crucial in what follows since it confirms that $\Phi^\eps$ is indeed a 
diffeomorphism preserving divergence-free vector fields.

\begin{lemma}
    \label{lem:div_preserve}
For every $\eps$ small enough, $\Phi^\eps$ is a diffeomorphism from $\mathscr{F}_S^0$ to $\mathscr{F}_{S^\eps}^0$.
\end{lemma}
\begin{proof}
Since $\varphi^\eps$ is an orientation preserving diffeomorphism, one has $J(\mu_S,\mu_S^\eps)\varphi^\eps>0$.
    Besides, 
    $$\Phi^{-\eps}(X)=\frac{1}{[J(\mu_S^\eps,\mu_S)\varphi^{-\eps}]\circ \varphi^\eps} D[(\operatorname{Id}+\eps \theta)^{-1}]X\circ \varphi^\eps,\qquad X\in \mathfrak{X}(S^{\eps}). 
    $$
    
    As a consequence, $\Phi^\eps$ defines a diffeomorphism from $\mathscr{F}_S$ to $\mathscr{F}_{S^\eps}$.
 We are left to prove that it preserves divergence-free vector fields.
According to the Hodge decomposition (see Appendix~\ref{append:Hodge}), it is enough to check that $ \Phi^\eps(\mathscr{F}_S)$ is orthogonal to $\{\nabla_{S^\eps} f \mid f\in \mathscr{C}^\infty(S^\eps)\}$.
Using the change of variables formula (cf.~\eqref{cdvIntegralSurf}), one has, for every $X\in \mathfrak{X}(S)$,
    \begin{align*}
        \int_{S^\eps} \langle df,\Phi^\eps(X)\rangle d\mu_{S^\eps}&=
        \int_{S^\eps} \langle df,\varphi^{\eps}_*(X)\rangle \frac{1}{[J(\mu_S,\mu_S^\eps)\varphi^\eps]\circ \varphi^{-\eps}}d\mu_{S^\eps}=\int_{S} \langle \varphi^{\eps,*}df,X\rangle d\mu_{S}\\
        &=\int_{S} \langle d(f\circ \varphi^{\eps}),X\rangle d\mu_{S}, 
    \end{align*}
    where the notation $\varphi^{\eps,*}$ stands for the conormal derivative of $\varphi^{\eps}$.
    Then $X$ is divergence-free if and only if $\Phi^\eps(X)$ is. 
The lemma is thus proved.
\end{proof}
\paragraph{Step 2: computation of the variation of $j$.}
Since we 
prefer to avoid dealing
with operators 
defined on $S^\eps$, we will use $\Phi^\eps$ to relate $\mathscr{F}^0_S$ and $\mathscr{F}_{S^\eps}^0$.

Let us first compute $(\Phi^\eps)^\dagger$. Let $j\in \mathscr{F}_S$ and $g\in \mathscr{F}_{S^\eps}$. One has
\begin{align*}
    \langle \Phi^\eps j , g \rangle &= \int_{S^\eps} \frac{1}{[J(\mu_S,\mu_S^\eps)\varphi^\eps]\circ \varphi^{-\eps}} \langle (\operatorname{Id}+\eps D\theta)j(\varphi^{-\eps}(x)), g(x) \rangle d\mu_{S^\eps}(x)\\
    &=\int_S \langle (\operatorname{Id}+\eps D\theta)j(x), g(\varphi^{\eps}(x))\rangle d\mu_S(x)\\
    &=\int_S \langle j(x), (\operatorname{Id}+\eps D\theta)^T g(\varphi^{\eps}(x))\rangle d\mu_S(x).
\end{align*}

We thus infer that $ (\Phi^\eps)^\dagger$ is given by
\begin{align*}
    (\Phi^\eps)^\dagger : \mathscr{F}_{S^\eps} & \longrightarrow \mathscr{F}_S\\
        g &\longmapsto  (\operatorname{Id}+\eps D\theta^T)g\circ \varphi^{\eps}.
\end{align*}

Let $j^\eps:=\Phi^{-\eps} (j_{S^\eps})$. According to Lemmas~\ref{lem:existence} and \ref{lem:div_preserve}, $j^\eps$ is well defined and belongs to $\mathscr{F}_S^0$.
To compute the differential of $j^\eps$, it is convenient to introduce the operators
$$
    Q^\eps:
    \begin{array}[t]{rcl}
         \mathscr{F}_S^0 &\longrightarrow &\mathscr{F}_S^0\\
        j&\longmapsto& (\Phi^\eps)^\dagger \Phi^\eps j
    \end{array} \quad \text{and}\quad
    L_\eps:
    \begin{array}[t]{rcl}
        \mathscr{F}_S^0 &\longrightarrow & L^2(P,\R^3)\\
        j &\longmapsto &\operatorname{BS}_{S^\eps} \Phi^\eps j
    \end{array}
$$
so that 
\begin{equation*}
 \forall j,k \in \mathscr{F}_S, \quad \Vert\Phi ^\eps(j)\Vert_{\mathscr{F}_ {S^\eps}}^2=\langle j, Q^\eps j \rangle_{\mathscr{F}_S} \quad\text{and}\quad  
     \langle Q^\eps j,  k \rangle_{\mathscr{F}_S}=\langle j, Q^\eps k \rangle_{\mathscr{F}_S}.
\end{equation*}

According to the optimality condition \eqref{cionOptimjS} on $j_{S^\eps}$, 
 $j^\eps$ is uniquely characterized by the identity
$$
\forall v \in \mathscr{F}_S^0, \quad     0= \langle L_\eps v, L_\eps j^\eps-B_T \rangle_{L^2(P,\R^3)}+ \lambda \langle v , Q^\eps j^\eps \rangle_{\mathscr{F}_S^0}
$$
which also rewrites
$$
\forall v \in \mathscr{F}_S^0, \quad 0= \langle v , \lambda Q^\eps j^\eps+ L_\eps^\dagger(L_\eps j^\eps-B_T) \rangle_{\mathscr{F}_S^0}.
$$
It follows that
\begin{align}
    \label{eq:j_eps}
    j^\eps=(\lambda Q^\eps+L_\eps^\dagger L_\eps)^{-1}L_\eps^\dagger B_T.
\end{align}
Let us now compute the first order variation of $j^\eps$. To this aim, we use the  expansion 
\begin{equation}\label{expand:Jacobien}
J(\mu_S,\mu_S^\eps)\varphi^\eps=1+\eps \dive_S \theta+\operatorname{o}(\eps)
\end{equation}
obtained in
\cite[Lemma~5.4.15]{henrotShapeVariationOptimization2018}. Recall that the notation $\dive_S \theta$ stands for the tangential divergence
of $\theta$ on $S$.
\begin{lemma}
\label{lem:operator_der}
Let $S$ and $\theta$ be chosen as above. Then, one has 
\begin{align*}
L_\eps&= \operatorname{BS}_S+ \left.\frac{dL_\eps}{d\eps} \right|_{\eps=0} \eps+\operatorname{o}(\eps)\text{ in } \mathcal{L}(\mathscr{F}_S^0,L^2(P,\R^3)), \\
Q^\eps &= I+ \left.\frac{dQ^\eps}{d\eps} \right|_{\eps=0}\eps+\operatorname{o}(\eps)\text{ in }\mathcal{L}(\mathscr{F}_S^0),
\end{align*}
where, for every $j \in \mathscr{F}_S^0$ and $y\in P$, 
\begin{align}
    \left(\left.\frac{dL_\eps}{d\eps} \right|_{\eps=0}j\right)(y)&= \int_S \left( K(x,y)\times(D\theta(x) j(x)) + (D_x K(x,y)\theta(x))\times j(x) \right) d\mu_S(x),\nonumber\\
    \left.\frac{dQ^\eps}{d\eps}\right|_{\eps=0}&=  D\theta+D\theta^T - \dive_S \theta \operatorname{Id}=e(\theta)- \dive_S \theta \operatorname{Id},\label{eq:**}
\end{align}
and $e(\theta)$ is defined as in \eqref{eq:*}. 
\end{lemma}

\begin{proof}[Proof of Lemma~\ref{lem:operator_der}]
    Let us start with $L_\eps$.
    Given $j \in \mathscr{F}_S^0$  and $y\in P$, we have
    \begin{align}
        L_\eps (j)(y)&=\int_{S^\eps} \frac{1}{[J(\mu_S,\mu_S^\eps)\varphi^\eps]\circ \varphi^{-\eps}}K(x,y)\times[(\operatorname{Id}+\eps D\theta)j(\varphi^{-\eps}(x))] d\mu_{S^\eps}(x)\nonumber\\
        &= \int_{S} K(\varphi^\eps(x),y)\times[(\operatorname{Id}+\eps D\theta)j(x)] d\mu_{S}(x)\nonumber\\
        &= \operatorname{BS}_S(j)(y) + \eps \int_S \left( K(x,y)\times (D\theta(x) j(x)) + [D_x K(x,y) \theta(x)]\times j(x) \right) d\mu_S(x) +\operatorname{o}(\eps).\label{eqintermLeps}
    \end{align}
Moreover, it can be easily checked that the reminder term of this expansion 
grows at most linearly with respect to $\Vert j\Vert_{\mathscr{F}_S}$. 
Regarding $Q^\eps$, a similar reasoning using \eqref{expand:Jacobien} yields 
    \begin{eqnarray*}
        Q^\eps&=&\frac{1}{[J(\mu_S,\mu_S^\eps)\varphi^\eps]\circ \varphi^{-\eps}} (\operatorname{Id}+\eps D\theta^T)(\operatorname{Id}+\eps D\theta)\\
        &=&\operatorname{Id}+\eps(D\theta+D\theta^T - \dive_S \theta \operatorname{Id}) +\operatorname{o}(\eps),
    \end{eqnarray*}
    concluding the proof.
\end{proof}

Combining all the results above, we now compute the sensitivity of $j^\eps$ with respect to $\eps$.   The following result is an immediate consequence of Lemma~\ref{lem:operator_der} and \eqref{eq:j_eps}.
\begin{proposition}\label{prop:jeps}
One has  $j^\eps=j_S+\left.\frac{dj^\eps}{d\eps}\right|_{\eps=0}\eps+\operatorname{o}(\eps)$
    with
    \begin{eqnarray*}
  \left.\frac{dj^\eps}{d\eps}\right|_{\eps=0} &=&
       \left(\lambda \operatorname{Id}+\operatorname{BS}_S^\dagger \operatorname{BS}_S \right) ^{-1} \left.\frac{d L_{\eps}^\dagger}{d\eps} \right|_{\eps=0}B_T \\
    && - \left(\lambda \operatorname{Id}+\operatorname{BS}_S^\dagger \operatorname{BS}_S\right) ^{-1}
    \left( \lambda \left.\frac{dQ^\eps}{d\eps}\right|_{\eps=0}+\left.\frac{d L_{\eps}^\dagger}{d\eps}\right|_{\eps=0}\operatorname{BS}_S +\operatorname{BS}_S^\dagger \left.\frac{d L_{\eps}}{d\eps}\right|_{\eps=0}\right)
    \left( \lambda \operatorname{Id}+\operatorname{BS}_S^\dagger \operatorname{BS}_S \right)^{-1} \operatorname{BS}_S^\dagger B_T .
    \end{eqnarray*}
\end{proposition}

\paragraph{Step 3: computation of the cost functional derivative.}
Recall that
\begin{eqnarray}
C(S^\eps)&=&\|\operatorname{BS}_{S^\eps} j_{S^\eps} -B_T \|^2_{L^2(P,\R^3)} + \lambda \|j_{S^\eps}\|^2_{\mathscr{F}_{S^\eps}}\nonumber \\
&=& \|L_{\eps} j^{\eps} -B_T \|^2_{L^2(P,\R^3)} + \lambda \langle j^\eps, Q^\eps j^\eps \rangle_{\mathscr{F}_S}\label{workExprCSeps}.
\end{eqnarray}
By differentiating this expression and according to Proposition~\ref{prop:jeps}, we get 
\begin{eqnarray*}
   \left. \frac{dC(S^\eps)}{d\eps}\right|_{\eps=0}&=&\lambda \left( \left\langle j_S, \left.\frac{dQ^\eps}{d\eps}\right|_{\eps=0} j_S \right\rangle_{\mathscr{F}_S}+ 2\left\langle j_S, \left.\frac{dj^\eps}{d\eps}\right|_{\eps=0}\right\rangle_{\mathscr{F}_S} \right)\\
   && +2\left\langle \operatorname{BS}_S j_S- B_T , \left.\frac{dL_\eps}{d\eps}\right|_{\eps=0}j_S +\operatorname{BS}_S \left.\frac{dj^\eps}{d\eps}\right|_{\eps=0}\right\rangle_{L^2(P,\R^3)}.
\end{eqnarray*}
Note that
\begin{align*}
2 \left\langle (\lambda \operatorname{Id}+ \operatorname{BS}_S^\dag \operatorname{BS}_S) j_S- \operatorname{BS}_S^\dag B_T
, \left.\frac{dj^\eps}{d\eps}\right|_{\eps=0}\right\rangle_{\mathscr{F}_S}=0. 
\end{align*}
Thus 
\begin{equation}
    \label{eq:Ii}
    \left. \frac{dC(S^\eps)}{d\eps}\right|_{\eps=0}= \lambda \left\langle j_S, \left.\frac{dQ^\eps}{d\eps}\right|_{\eps=0} j_S \right\rangle_{\mathscr{F}_S}+2\left\langle \operatorname{BS}_S j_S- B_T , \left.\frac{dL_\eps}{d\eps}\right|_{\eps=0}j_S \right\rangle_{L^2(P,\R^3)}.
\end{equation}
\begin{remark}
The previous expression can be understood as follows: 
writing $C(S)=:\tilde C(S,j_S)$ with the natural choice of $\tilde C$ 
and assuming that $(C,j)\mapsto \tilde C$ and $S\mapsto j_S$ are sufficiently regular, one has
    $$
    \frac{\partial \tilde{C}(S,j_S)}{\partial S}= \frac{\partial \tilde C}{\partial S}(S,j_S)+ \frac{\partial \tilde C}{\partial j}\frac{\partial j_S}{\partial S}(S,j_S).
    $$
    Using the fact that $\frac{\partial \tilde C}{\partial j}(j_S)=0$ since $j_S$ is the minimizer of $j \mapsto \tilde{C}(S,j)$, we get 
    $$
    \frac{\partial \tilde{C}(S,j_S)}{\partial S}= \frac{\partial \tilde C}{\partial S}(S,j_S).
    $$
\end{remark}

In what follows, we will use the identity stated in the following lemma.
\begin{lemma}\label{lem:identitdLepsdeps}
Let $j\in \mathcal{F}_S^0$, $k\in L^2(P,\R^3)$, and $\theta$ be as in the statement of Theorem~\ref{theo:shapeDer}. Then 
$$
\left\langle \left.\frac{dL_\eps}{d\eps}\right|_{\eps=0}j,k\right\rangle_{L^2(P,\R^3)}=-\left\langle D\theta j,Z_P(k)\right\rangle_{\mathcal{F}_S}-\left\langle \theta,\widehat{Z}_P(k,j)\right\rangle_{\mathcal{F}_S}.
$$
\end{lemma}
\begin{proof}
The proof follows from straightforward computations, by combining the Fubini theorem with standard properties of the scalar triple product\footnote{Recall that 
 the scalar triple product  of three vectors $a,b,c\in \R^3$ is given by $\langle{a} ,({b} \times {c} )\rangle$ and coincides with the (signed) volume of the parallelepiped defined by the three vectors. Therefore, the scalar triple product is preserved by a circular shift of the triple $(a,b,c)$.}.
\end{proof}

By combining \eqref{eq:**}, \eqref{eq:Ii}, and Lemma~\ref{lem:identitdLepsdeps}, one computes
\begin{align*}
    \left. \frac{dC(S^\eps)}{d\eps}\right|_{\eps=0}&  ={}&  \lambda \left\langle j_S, e(\theta) j_S-\dive_S \theta j_S \right\rangle_{\mathscr{F}_S}
-2\left\langle D\theta j_S ,Z_P (\operatorname{BS}_S j_S- B_T) \right\rangle_{\mathscr{F}_S}
 -2\left\langle \theta , \widehat{Z}_P(\operatorname{BS}_S j_S- B_T,j_S)\right\rangle_{\mathscr{F}_S}.
\end{align*}
    
To conclude this computation, observe that for all vectors $u$ and $v$ in $\R^3$,
$$
\langle D\theta u, v\rangle =\sum_{i,j=1}^3(D\theta)_{ij}u_jv_i=D\theta :(uv^T), \quad  \langle (D\theta)^T u, v\rangle=D\theta:(vu^T), \quad \langle e(\theta)u, v\rangle=D\theta :(uv^T+vu^T)
$$ 
so that 
$$
\dive_S\theta=\sum_{i=1}^3\partial_{x_i}\theta_i-D\theta :(\nu\nu^T)=D\theta:(I_3-\nu\nu^T).
$$
We thus obtain
$$
\left. \frac{dC(S^\eps)}{d\eps}\right|_{\eps=0}=\int_S (\langle \theta , X_1\rangle  +D \theta : X_2 )  \, d\mu_S,
$$
where $X_1$ and $X_2$ have been introduced in the statement of the theorem.

Now, according to \cite[Prop.~5.4.9]{henrotShapeVariationOptimization2018}, the shape differential $\langle dC(S),\theta\rangle$ above can be recast as
\begin{align*}
\langle dC(S),\theta\rangle &=  \int_S \langle \theta , X_1\rangle  +	\sum_{i=1}^3\int_S\frac{\partial \theta_i}{\partial \nu}\langle  (X_2)_{i:}, \nu\rangle    \, d\mu_S+	\sum_{i=1}^3\int_S \langle\nabla_S \theta_i, [(X_2)_{i:}]_S\rangle   \, d\mu_S\\
&=  \int_S \langle\theta , X_1\rangle  +	\sum_{i=1}^3\int_S \frac{\partial \theta_i}{\partial \nu}\langle (X_2)_{i:}, \nu \rangle   \, d\mu_S+	\sum_{i=1}^3\int_S \theta_i \left(-\operatorname{div}_S[(X_2)_{i:}]_S+\langle\kappa[(X_2)_{i:}]_S, \nu\rangle \right)   \, d\mu_S,
\end{align*}
where, for $i\in \{1,2,3\}$, $(X_2)_{i:}$ denotes the $i$-th line of $X_2$ seen as a column vector, $[(X_2)_{i:}]_S$ is the tangential part of $(X_2)_{i:}$ defined as $[(X_2)_{i:}]_S=(X_2)_{i:}-\langle[(X_2)_{i:}]_S, \nu\rangle $, and $\kappa$ denotes the mean curvature\footnote{The mean curvature  $\kappa$ of a surface is defined here as the sum of the principal curvatures of $S$.} on $S$. The expected formula is obtained by noting that each line of $X_2$ is tangential (in other words, normal to $\nu$). Indeed, this follows from the definitions of the mapping $Z_P$, the function $j_S$, and the fact that $I_3-\nu\nu^T$ corresponds to the matrix of the orthogonal projection onto $S$.

To conclude this proof, it remains to investigate the shape differentiability of $S\mapsto C(S)$. 
Let us introduce $\tau_\theta=\operatorname{Id}+\theta$ where $\theta$ is chosen as in the statement of the theorem. It is straightforward to show that the real number $C(\tau_\theta(S))$ can be written as a 
smooth function of integrals written on the fixed domain $S$, for which the integrand depends regularly on $\theta$.
Indeed, this can be straightforwardly obtained by replacing $\operatorname{Id}+\eps\theta$ by $\operatorname{Id}+\theta$ in the reasoning above, and mimicking the associated computations leading to \eqref{eq:j_eps}, \eqref{eqintermLeps} and \eqref{workExprCSeps}. 
This yields to the expansion
$$
C(\tau_\theta(S))=C(S)+   \left. \frac{dC(S^\eps)}{d\eps}\right|_{\eps=0}+\operatorname{o}(\Vert \theta\Vert_{W^{2,\infty}(\R^d,\R^d)}),
$$
with $S^\eps=(\operatorname{Id}+\eps \theta)S$ and the shape differentiability of $C$ hence follows.

\section{Numerical implementation}\label{sec:num}
The results obtained above are intrinsic in the sense that they do not depend on the parametrization of the objects (surfaces, magnetic field, electric current, \ldots).
There are several ways to represent them numerically. We have chosen to use what is, to the best of the authors' knowledge, the classical approach in the stellarator community.
In particular:
\begin{itemize}
    \item Surfaces, vector fields and magnetic fields are represented by Fourier coefficients. We detail the parametrization in Section \ref{subsection:repr}.
    \item Stellarator symmetry is imposed on all the objects. We refer to \cite[Section 12.3]{imbert-gerardIntroductionSymmetriesStellarators2019} and \cite{dewarStellaratorSymmetry1998} for details and justifications of the stellarator symmetry.
    \item As mentioned in Remark~\ref{rmk:LS_simplifications}, $\operatorname{BS}_S$ is slightly modified. Not only the optimization space  $\mathscr{F}_S^0$ is replaced by a suitable affine subspace of it, but also we restrict the image of $\operatorname{BS}_S$ to the plasma boundary. We provide further details in \Cref{subsec:current-sheet_representation,subsec:Magnetic_field_representation} and \Cref{subsec:poisson_torus}.
\end{itemize}

\subsection{Parametrization issues}
\label{subsection:repr}
\subsubsection{Surface representation}

We represent a toroidal surfaces $S$ as the image of the two-dimensional flat torus $T=(\R/\Z)^2$ by an
embedding
\begin{align*}
    \psi :\quad T &\to \R^3\\
    (u,v) &\mapsto \psi(u,v).
\end{align*}

Stellarators often exhibits a discrete symmetry by rotation. For example W7X is invariant by the rotation of angle $2\pi/5$ along the vertical axis and NCSX has an invariance by the rotation of angle $2\pi/3$.
To reduce the complexity, we 
 only represent one module of the surface and we denote by $N_p$ the number of modules needed to generate the entire surface (using rotations of angle $2 \pi /N_p$).
We introduce the cylindrical coordinates  $(R,\varphi,Z)$.
We will make the assumption of no toroidal folding, i.e., that the intersection of each half plane $ \{\varphi=constant \}$ with $S$ is a single loop.
We express $\psi$ in cylindrical coordinates $(R(u,v),\frac{2\pi v}{N_p},Z(u,v))$ as 
\[ 
    \begin{pmatrix}
        x\\y\\z
    \end{pmatrix}
    =
    \begin{pmatrix}
        R(u,v) \cos(\frac{2\pi v}{N_p})\\
        R(u,v) \sin(\frac{2\pi v}{N_p})\\
        Z(u,v)
    \end{pmatrix},\qquad (u,v)\in T.
\]
Then we develop $R$ and $Z$ in Fourier components and we impose the stellarator-symmetry
\begin{align}
    R(u,v)&=\sum_{m \geq 0} \sum_{n\in \Z} R_{m,n}\cos(2 \pi (m u+n v)) ,\label{eq:R_in_fourier}\\
    Z(u,v)&=\sum_{m \geq 0} \sum_{n\in \Z} Z_{m,n}\sin(2 \pi (m u+n v)) .
    \label{eq:Z_in_fourier}
\end{align}

Note the absence of $\sin$ terms for $R$ and $\cos$ terms for $Z$.
For the numerical simulation, we truncate the number of Fourier components in \eqref{eq:R_in_fourier} and \eqref{eq:Z_in_fourier}.

\begin{remark}
    The cost considered in this paper only depends on the surface (and is independent of its parametrization $\psi$).
On the toroidal direction, we have already imposed that $\varphi = 2 \pi \ms{v}/N_p$.
On the other hand, we can compose $\psi$ with any diffeomorphism $f_v:\R/\Z \to \R/\Z$ on the poloidal direction $u$.
Namely, $\psi(f_v(u),v)$ and $\psi(u,v)$ have the same image for a fixed $v$.
Thus our problem is invariant under the action of a smooth family of 
diffeomorphisms.
This extra degree of freedom has two consequences:
    \begin{itemize}
        \item If we use a regular discretization for the surface $(\psi(\frac{i}{n_u},\frac{j}{n_v}))_{i,j}$ of size $n_u\times n_v$, we need $|\partial_u \psi|$ and $|\partial_v \psi|$ to be as regular as possible.
        \item As we take a finite number of harmonics, we would like to ``compress" as much as possible the information on the shape by using low harmonics.
    \end{itemize}
This problem has been study in the plasma community in \cite{hirshman_explicit_1998} and gave rise to the notion of spectrally optimized Fourier series.
Nevertheless, we would like to highlight that this approach is extrinsic (it depends on the parametrization) and should not be used for other purposes than fixing the gauge invariance. We have not implemented it since our numerical results empirically already provided a reasonable regularity on the poloidal parametrization.
\end{remark}

\subsubsection{Magnetic field representation}
\label{subsec:Magnetic_field_representation}
In the previous sections, we represented the target magnetic field as a three-dimensional vector field in the plasma domain.
Nonetheless, thanks to the structure of Maxwell's equations, the magnetic field inside $P$ is nearly entirely determined by its normal component along $\partial P$.
It is thus possible to work with a scalar quantity on a surface instead of a three-dimensional vector field on a volume.
Indeed, let us introduce the line integral (also called circulation) of the magnetic field along one toroidal turn. By Stoke's theorem (also known as Ampere circuital law in electromagnetism), this is equal to the total flux of the electric current across any surface enclosed by the above-mentioned toroidal loop. This quantity is called the \emph{total poloidal current} and is denoted by $I_p$.

As proved in Appendix~\ref{subsec:poisson_torus}, $I_p$ and the normal component of the magnetic field across $\partial P$ characterize completely the magnetic field inside the plasma.

As a consequence, it is reasonable to minimize
\[
    \chi^2_B(j)=\int_{\partial P} \langle (\operatorname{BS}_S j- B_T), \nu  \rangle^2 d\mu_{\partial P}
\]
with the total poloidal current of $j$ fixed, where $\nu$ denotes the outward normal unit vector to $\partial P$.

This idea has been used by physicists for a long time, for example \cite{merkelSolutionStellaratorBoundary1987,landremanImprovedCurrentPotential2017}. Besides, if we consider two currents distribution $j$ and $\tilde{\jmath}$ on two toroidal surfaces $S$ and $\tilde S$ outside of $P$ with the same total poloidal 
currents, the induced magnetic field in $P$ satisfies
\[
    \|\operatorname{BS}_S j- \operatorname{BS}_{\tilde S} \tilde{\jmath}\|_{L^2(P,\R^3)} \lesssim \int_{\partial P} \langle (\operatorname{BS}_S j-\operatorname{BS}_{\tilde S} \tilde{\jmath}), \nu  \rangle^2 d\mu_{\partial P}.
\]
We provide mathematical proofs of these facts in Appendix \ref{subsec:poisson_torus}.

We also use a normal target magnetic field that respects the stellarator symmetry, that is,
\[\langle B_T, \nu \rangle(\psi(u,v))=\sum_{m \geq 0} \sum_{n\in \Z} B_{m,n}\sin(2\pi (m u +nv)).\]
As before, we truncate the Fourier series to obtain a numerically tractable expression.
\subsubsection{Current-sheet representation}
\label{subsec:current-sheet_representation}
As mentioned in the previous section, we need to parameterized all divergence-free vector field on $S$ with a fixed total poloidal current $I_p$.
In Appendix~\ref{subsec:div_free_on_surface} we prove that
\begin{align}
    \mathscr{F}_T^0=\{\nabla^\perp \Phi +\lambda_1 \partial_u+ \lambda_2 \partial_v \mid \Phi \in H^1(T), (\lambda_1,\lambda_2) \in \R^2\} \label{eq:div_free_flat_torus}
\end{align}
with $\nabla^\perp \Phi=\frac{\partial \Phi}{\partial u}\partial_v - \frac{\partial \Phi}{\partial v} \partial_u$.

The following lemma describes how embeddings induce isomorphisms between $\mathscr{F}_T^0$ and $\mathscr{F}_S^0$.

\begin{lemma}
Let $\psi : T \to  \R^3$ be an embedding with $S=\psi(T)$ and  consider
        \begin{align*}
            \Psi:  \mathfrak{X}(T)& \to \mathfrak{X}(S)\\
            X &\mapsto \frac{D\psi X}{\left|\frac{\partial \psi}{\partial u} \times \frac{\partial \psi}{\partial v}\right|}.
        \end{align*}
Then    $\Psi$ induces an isomorphism between $\mathscr{F}_T^0$ and $\mathscr{F}_S^0$.
\end{lemma}
The proof is completely similar to that of Lemma~\ref{lem:div_preserve}.

Let us suppose now that $(u,v)$ are poloidal and toroidal coordinates for the parameterization $\psi$, that is,
\begin{itemize}
    \item $\Gamma_p: \R/\Z\ni t\mapsto \psi(t,0)\in  S$ is a loop doing exactly one poloidal turn (and 0 toroidal ones);
    \item $\Gamma_t: \R/\Z \ni t\mapsto \psi(0,t)\in S$ is a loop doing exactly one toroidal turn (and 0 poloidal ones).
\end{itemize}
Besides, as is it in general the convention in the dedicated literature, we assume that $\psi$ is 
orientation reversing,
meaning that 
\begin{equation}\label{LaRe1301}
(\Psi(\partial_u),\Psi(\partial_v),-\nu)\text{ is direct,}
\end{equation} 
with $\nu$ the outward normal vector field.
\begin{lemma}
    Let $X=\nabla^\perp \Phi +I_p \partial_u+ I_t \partial_v$.
    Then the poloidal (respectively, toroidal) flux of $\Psi(X)$, i.e., the flux of $\Psi(X)$ across $\Gamma_t$ (respectively,  $\Gamma_p$), is given by $I_p$ (respectively, $I_t$).
\end{lemma}
\begin{proof}
        Remark that $\dive \Psi(X)=0$ ensure that the flux across any loop depends only on the isotopic class of the loop considered.
        Recall that the flux of $\Psi(X)$ across some loop $\Gamma$ is given by
        $$\oint_\Gamma \langle \Psi(X) , \left(\frac{\Gamma'}{|\Gamma'|}\times -\nu \right)\rangle d\mu_\Gamma=\int_0^1 \langle \Psi(X) , (\Gamma' \times -\nu )\rangle(\Gamma(t)) dt,$$
      where the choice of the sign in $-\nu$ is due to  to the convention \eqref{LaRe1301}.
        Thus, the flux across $\Gamma_t$ (the poloidal flux) is
        \begin{align*}
            \int_0^1 \langle \Psi(X) , (\frac{\partial \psi}{\partial v} \times -\nu)\rangle (\Gamma_t(t)) dt
                &=-\int_0^1 \langle \nu , (\Psi(X) \times \frac{\partial \psi}{\partial v})\rangle (\Gamma_t(t)) dt\\
                &=\int_0^1 \frac{1}{\left|\frac{\partial \psi}{\partial u} \times \frac{\partial \psi} {\partial v}\right|^2}\langle (\frac{\partial \psi}{\partial u} \times \frac{\partial \psi} {\partial v}) , (D\psi (\nabla^\perp \Phi +I_p\partial_u +I_t \partial_v) \times \frac{\partial \psi}{\partial v})\rangle (\Gamma_t(t)) dt\\
                &=\int_0^1 \frac{1}{\left|\frac{\partial \psi}{\partial u} \times \frac{\partial \psi} {\partial v}\right|^2}\langle (\frac{\partial \psi}{\partial u} \times \frac{\partial \psi} {\partial v}) , (I_p-\frac{\partial \Phi}{\partial v}) \frac{\partial \psi}{\partial u} \times \frac{\partial \psi}{\partial v})\rangle (\Gamma_t(t)) dt\\
                &=\int_0^1\left( I_p-\frac{\partial \Phi}{\partial v}(0,t) \right)dt\\
                &=I_p.
        \end{align*}
        The computation of the flux across $\Gamma_p$ is analogous.
\end{proof}
 Thanks to this lemma, in order to minimize
on $\mathscr{F}_S^0$ we 
fix $I_p$ and $I_t$ 
and minimize with respect to $\Phi$.
Indeed, $I_p$ is 
fixed by the toroidal circulation of the target magnetic field (see \ref{subsec:poisson_torus}), whereas
$I_t$ is usually set to 0. This second condition is necessary to ensure the existence of ``poloidal coils". Otherwise, no closed field lines would realize one poloidal turn and zero toroidal ones.
 Thus, the set of admissible currents is described by
 $$J_{\mathrm{adm}}(S)=\{ \Psi(\nabla^\perp \Phi+I_p \partial_u+I_t\partial_v) \mid \Phi \in H^1(T)\}.$$
We say that $\Phi$ is the \emph{scalar 
current potential}.
By stellarator symmetry, its expansion in Fourier series is 
 \begin{align*}
 \Phi(u,v)=\sum_{m \geq 0} \sum_{n\in \Z} \Phi_{m,n}\sin(2 \pi (m u+n v)).
\end{align*}

Let us denote 
\[j_{a,S}=\Psi(I_p \partial_u+I_t\partial_v)\qquad \mbox{and}\qquad \hat{\mathscr{F}}_S^0=\{ \Psi(\nabla^\perp \Phi) \mid \Phi \in H^1(T)\}.\]
It is straightforward that the affine decomposition 
$ J_{\mathrm{adm}}(S)=j_{a,S}+\hat{\mathscr{F}}_S^0$
is compatible with $\Phi^\eps$ (cf. \eqref{def:Phi^eps}),
meaning that for any shape deformation $\theta$,
\begin{align*}
    \Phi^\eps(j_{a,S})=j_{a,S^\eps} \qquad \text{and }\qquad  \Phi^\eps(\hat{\mathscr{F}}_S^0)=\hat{\mathscr{F}}_{S^\eps}^0.
\end{align*}
Thus, we can consider the restriction of $\operatorname{BS}_S$ to $\hat{\mathscr{F}}_S^0$ that we will denote $\widehat{\operatorname{BS}}_S$. Let $\widehat{\operatorname{BS}}_S^\dagger$ be its adjoint (in $\hat{\mathscr{F}}_S^0$)
and $\hat{\pi}_S$ the orthogonal projector defined in $\mathscr{F}_S^0$ onto $\hat{\mathscr{F}}_S^0$.
Then Lemma~\ref{lem:existence} holds and the expression of the unique minimizer is given by
\begin{align*}
    \hat{\jmath}_S &= (\lambda \operatorname{Id}+\widehat{\operatorname{BS}}_S^\dagger \widehat{\operatorname{BS}}_S)^{-1} \left( \widehat{\operatorname{BS}}_S^\dagger (B_T-\operatorname{BS}_S j^a_S)-\lambda \hat{\pi}_S j_S^a \right),\quad j_S=j_S^a +\hat{\jmath}_S\\
    C(S)&= \lambda\Vert j_S\Vert^2_{\mathscr{F}_S}+ \Vert \operatorname{BS}_S j_S -B_T \Vert^2_{L^2(P,\R^3)}.
\end{align*}
\subsection{Implementation}
We wrote our implementation in python using several scientific computing open source libraries and, in particular:
\begin{itemize}
     \item Numpy \cite{harris2020array} for array computation,
     \item Scipy \cite{2020SciPy-NMeth} for the implementation of the  Broyden-Fletcher-Goldfarb-Shanno (BFGS) minimization algorithm,
     \item Opt\_einsum \cite{Smith2018} for optimizing tensor construction,
     \item Dask \cite{dask} for large array and efficient scientific computing parallelization,
     \item Matplotlib \cite{Hunter:2007} and Mayavi \cite{ramachandran2011mayavi} for plotting and graphic representations.
\end{itemize}
The full code is available on our gitlab\footnote{\url{https://plmlab.math.cnrs.fr/rrobin/stellacode}} under MPL 2 license.

The constraints on the perimeter, the reach and the plasma-CWS distance are implemented as a nonlinear penalization cost which blows up rapidly once the values exceed (or subceed) a given threshold. We refer to the code documentation for further details\footnote{\url{https://rrobin.pages.math.cnrs.fr/stellacode/}}.

\subsection{Numerical results}
In what follows, the data used for the simulations come from the NCSX stellarator equilibrium known as LI383 \cite{Zarn}. We will also use as reference CWS the one used in the original REGCOIL paper \cite{landremanImprovedCurrentPotential2017}.

We present here four simulations. We used either $\lambda = 2.5e^{-16}$ or $\lambda = 5.1e^{-19}$ as regularization parameter in the expression of the cost $C$. We mesh the CWS and the plasma surface with $64\times 64$ grids.
The scalar current potential $\Phi$ is developed in Fourier series up to order 12 in both directions.
The optimization is performed with up to $2000$ steps of the BFGS algorithm.
In every simulation
we implemented a penalization on the perimeter of the CWS (penalization above $56$m$^2$) and on plasma-CWS distance (penalization under $20$cm). We also implemented a reach penalization for two simulations (penalization under $7.69$cm).
Let us call Ref 
the initial CWS.
We use DP 
to refer to the simulations with distance and perimeter penalization and DPR 
for those with additional reach penalization.
The numerical results are summarized in \Cref{fig:table_medium_reg,fig:table_small_reg}. Figure~\ref{fig:multi_plot} illustrates the convergence history of the implemented optimization algorithm.
\begin{table}[h!]
    \begin{align*}
        \begin{array}{|c|c|c|c|c|c|c|c|c|c}
            \hline
            type&\chi_B^2&\chi_j^2&C(S)&\text{Distance } (m)&\text{Perimeter }(m^2)&\text{Reach }(m)&\text{number of iteration}\\
            \hline
            \text{Ref}&4.80e^{-03}&1.43e^{+14}&4.06e^{-02}&1.92e^{-01}&5.57e^{+01}&8.40e^{-02}&\\ 
            \hline
            \text{DPR}&1.23e^{-03}&9.48e^{+13}&2.49e^{-02}&1.99e^{-01}&5.60e^{+01}&7.69e^{-02}&775\\ 
            \hline
            \text{DP}&1.05e^{-03}&7.36e^{+13}&1.95e^{-02}&2.00e^{-01}&5.60e^{+01}&4.33e^{-06}&2000\\ 
            \hline
        \end{array}
    \end{align*}
    
    \caption{Numerical results for $\lambda=2.5e^{-16}$}
    \label{fig:table_medium_reg}
\end{table}

\begin{table}[h!]
    \begin{align*}
        \begin{array}{|c|c|c|c|c|c|c|c|c|c}
            \hline
            type&\chi_B^2&\chi_j^2&C(S)&\text{Distance } (m)&\text{Perimeter }(m^2)&\text{Reach }(m)&\text{number of iteration}\\
            \hline
            \text{Ref}&1.44e^{-04}&4.91e+14&3.94e^{-04}&1.92e^{-01}&5.57e^{+01}&8.40e^{-02}&\\ 
            \hline
            \text{DPR}&9.05e^{-06}&1.26e+14&7.34e^{-05}&2.00e^{-01}&4.17e+01&7.69e^{-02}&2000\\
            \hline
            \text{DP}&7.16e^{-06}&1.21e+14&6.90e^{-05}&2.00e^{-01}&5.60e+01&8.33e^{-05}&2000\\
            \hline 
        \end{array}
    \end{align*}
    \caption{Numerical results for $\lambda=5.1e^{-19}$} 
    \label{fig:table_small_reg}
\end{table}
\begin{remark}
    Without penalization on the reach, one naturally obtains better results (as less constraints are applied on the set of admissible shapes).
    Nevertheless, such an approach seems a very bad idea:
    \begin{itemize}
        \item theoretically, since the existence of an optimal shape is guaranteed only for bounded reach,
        \item numerically, because sharper and sharper ``spikes" appear, as shown in \Cref{fig:spikes}. Those spikes can be arbitrary long while still keeping a finite perimeter (and encapsulated volume).
    \end{itemize}
\end{remark}

\begin{center}
\begin{figure}[h]
\begin{minipage}{8cm}
        \centering
            \includegraphics[scale=0.6]{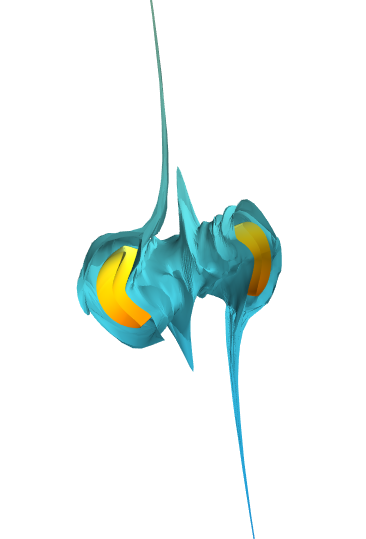}
          \caption{Main pattern of the optimal CWS for the DP simulation with $\lambda=2.5e^{-16}$, top and bottom spikes have been truncated. To obtain the complete CWS, it is enough to make two rotations of angle $2\pi/3$ around the principal axis of the stellarator. \label{fig:spikes}}
    \end{minipage} \hspace{0.5cm}
    \begin{minipage}{8cm}
        \centering
    \begin{minipage}{8cm}
        \centering
        \includegraphics[scale=0.17]{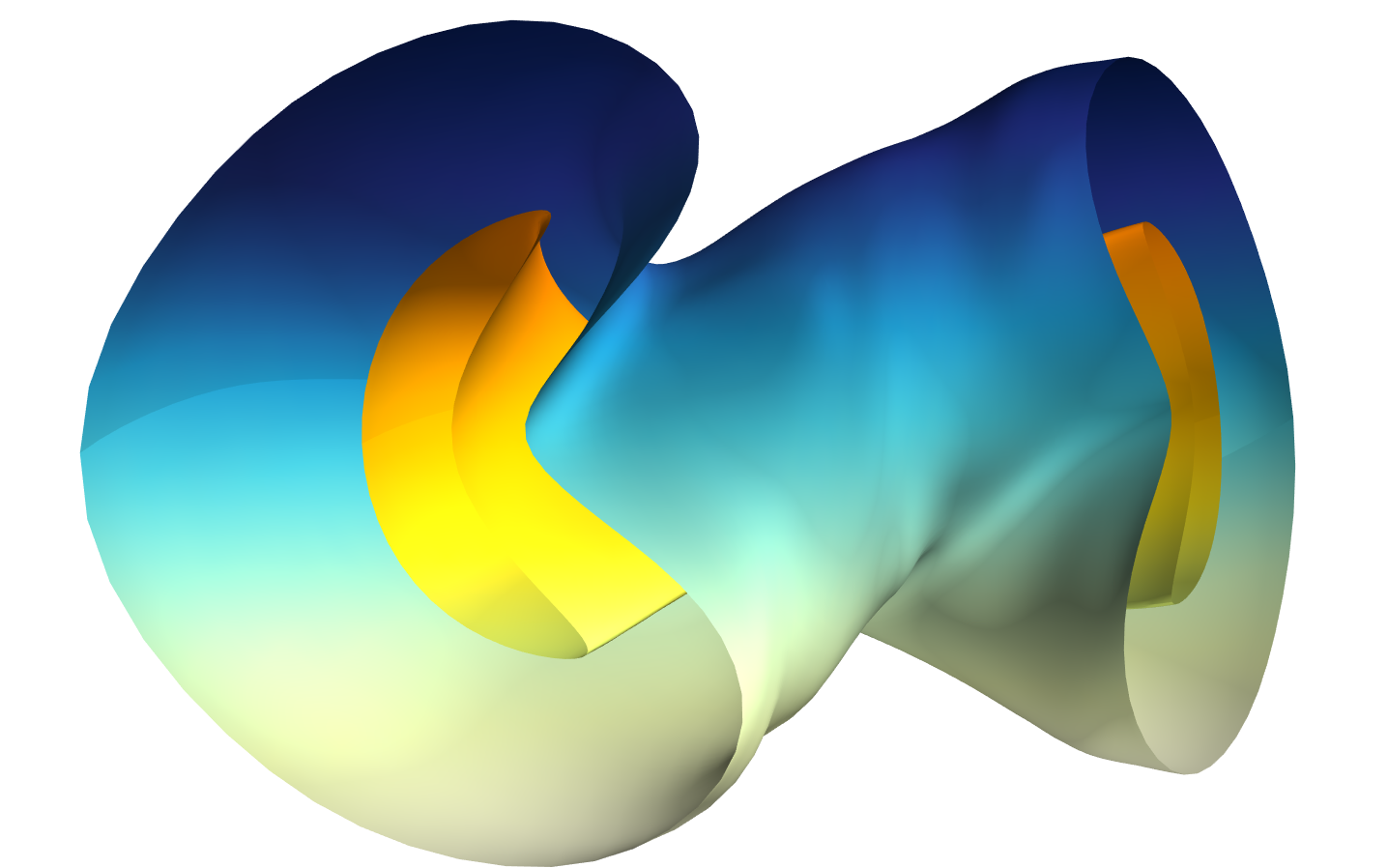}
        \caption{Main pattern of the CWS (blue and white) and plasma surface (orange) of the National Compact Stellarator Experiment (NCSX) designed by the Princeton Plasma Physics Laboratory. To obtain the complete CWS, it is enough to make two rotations of angle $2\pi/3$ around the principal axis of the stellarator}
    \end{minipage}
    \begin{minipage}{8cm}
        \centering
     \includegraphics[scale=0.26]{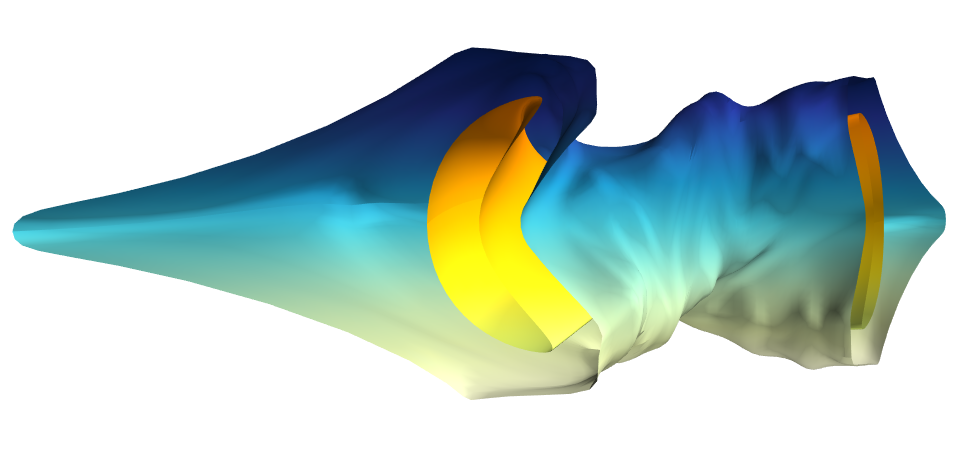}
    \caption{Main pattern of the optimal CWS for the DPR simulation with $\lambda=2.5e^{-16}$. To obtain the complete CWS, it is enough to make two rotations of angle $2\pi/3$ around the principal axis of the stellarator.}
    \end{minipage}
        \end{minipage}     
\end{figure}
\end{center}
%
%

\begin{figure}
    \includegraphics[width=\textwidth]{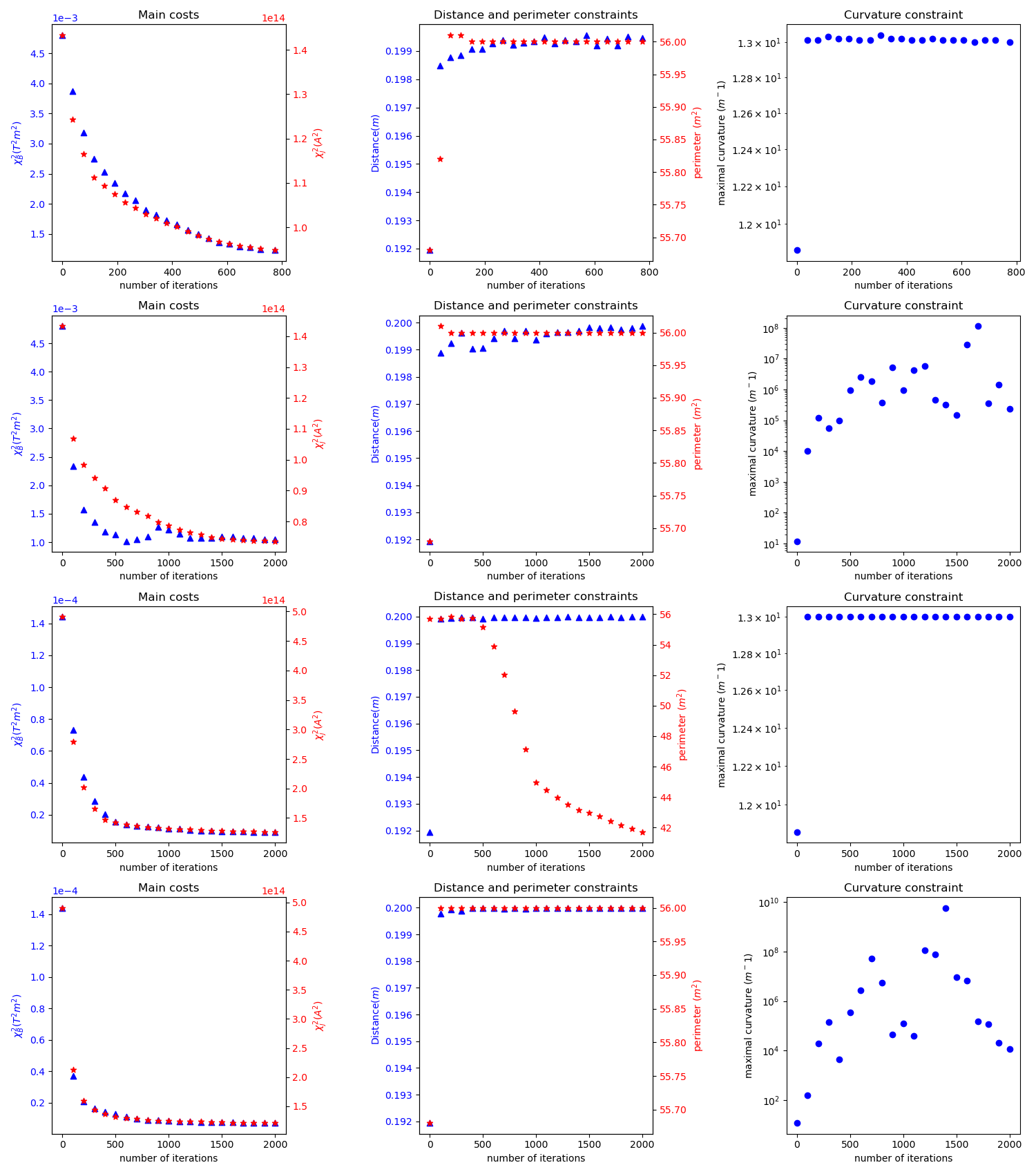}
    \caption{History of convergence for the implemented optimization algorithm. From left to right, evolution of the costs (left), distance and perimeter constraints (middle) and  the curvature constraint (right) along the optimization process. From top to bottom: configurations DPR and DP, respectively for  Tables~\ref{fig:table_medium_reg} ($\lambda=2.5e^{-16}$) and \ref{fig:table_small_reg} ($\lambda=5e^{-19}$).}
    \label{fig:multi_plot}
\end{figure}

\appendix
\section{Some differential geometry}
In this section, we recall some basics fact about differential geometry and vector fields on toroidal surfaces and domains.
\subsection{Hodge decomposition}\label{append:Hodge}
We recall in this part some
notions of differential geometry and in particular of Hodge theory. We refer to
\cite[Chapter 3]{jostRiemannianGeometryGeometric2017} and \cite{leeIntroductionSmoothManifolds2012} for details and precise definitions  in the smooth setting. Although we are only interested in $\mathscr{C}^{1,1}$ manifolds in this article, note for the sake of completeness that details on Hodge theory for Lipschitz manifolds can be found for instance in \cite{PMIHES_1983__58__39_0}.

The Hodge decomposition is a powerful tool which gives an orthogonal decomposition of the space of square integrable $p$-forms on a Riemannian closed manifold $M$ as 
$$L^2_p(M)=B_p \oplus B^*_p \oplus \mathcal{H}_p, $$
where $B_p$ is the $L^2$-closure of $\{d \alpha \mid \alpha \in \Omega^{p-1}(M) \}$, $B_p^*$ is the $L^2$-closure of $\{d^* \beta \mid \beta \in \Omega^{p+1}(M) \}$ ($d^*$  is the coderivative), and $\mathcal{H}_p$ is the set $\{ \omega \in \Omega^{p}(M) \mid \Delta_H \omega=0\}$ of harmonic $p$-forms  with $\Delta_H$ the Hodge Laplacian. 

We  apply this result to the simple case of 1-forms on a two-dimensional closed Riemannian manifold $S$.
We recall a few basics facts:
\begin{itemize}
    \item 1-forms and vector fields can be identified thanks to the Riemannian metric. This isomorphism is called the \emph{musical isomorphism} and we denote by $X^b$ the 1-form defined as the image of a vector field $X$ thought the musical isomorphism. Conversely $w^{\#}$ denotes the vector field which is the image of the 1-form $\omega$.
    \item The divergence of a vector field $X$ is $-d^*X^b$.
    \item $d\circ d=0$ and $d^*\circ d^*=0$.
    \item $\Delta_H \alpha=0$ is equivalent to the system of equations $\begin{cases}d\alpha=0,\\d^* \alpha=0.
    \end{cases}$
\end{itemize} 
We want to show that the space of ``divergence-free" 1-forms (i.e., 
$\ker d^*$) coincides with 
$B^*_1\oplus \mathcal{H}_p$. It is clear that the latter space is contained in $\ker d^*|_{\Omega^1({M})}$. Conversely, for every exact form $\omega$, i.e., such that $\omega=df$ with $f\in \mathscr{C}^\infty(M)$, one has $d^*\omega=d^*df=\Delta_H f$. We recall that the Hodge Laplacian coincides with the Laplace--Beltrami operator on 0-forms. But $\Delta_H f=0$ implies that $f$ is constant on each connected component, thus $d^*\omega=0$ implies that $\omega=0$. As a result the space of divergence-free 1-forms is $B^*_1\oplus \mathcal{H}_p= (B_1)^\perp$.

Equivalently, 
the space of divergence-free vector fields coincides with 
the orthogonal to $\{\nabla f\mid f\in \Omega^0(M)\}$. In Appendix~\ref{subsec:div_free_on_surface} we give an explicit description of $(B^*_1\oplus \mathcal{H}_p)^{\#}$ for the two-dimensional flat torus.

\subsection{Divergence-free vector field on a flat torus}
\label{subsec:div_free_on_surface}
            Let $T=(\R/\Z)^2$ be the flat torus with Cartesian parametrization $(u,v)$.
            We want to characterize the set of divergence-free vector fields on $T$.

            As explained in \ref{append:Hodge}, we only need to characterizes $B_1^*(T)$ and $\mathcal{H}_1(T)$.
            \begin{itemize}
                \item $B_1^*(T)$ is the $L^2$-closure of the 1-forms
                $\frac{\partial \Phi}{\partial u}dv-\frac{\partial \Phi}{\partial v}du$ for $\Phi\in \mathscr{C}^\infty(T)$.
                \item $\mathcal{H}_1(T)$ is a two-dimensional vector space as the first Betti number for a 2D torus satisfies $b_1=2$. We easily compute $\mathcal{H}_1(T)= \{ \lambda_1 du+\lambda_2 dv \mid (\lambda_1,\lambda_2)\in \R^2 \}$.
            \end{itemize}
            Using the musical isomorphism, we deduce that all divergence-free vector fields in $L^2$ have the form given in Equation \eqref{eq:div_free_flat_torus}.
            

\subsection{Poisson equation on a toroidal 3D domain}
\label{subsec:poisson_torus}
    Given a toroidal 3D domain $P$, we want to study the Maxwell equations in vacuum inside $P$. We introduce a toroidal loop  $\Gamma$ inside $P$ and denote by $I_p$ the electric current-flux across any surface enclosed by $\Gamma$. By the conservation of charges ($\dive j=0$), this quantity is well defined. By smoothness of the Biot and Savart operator and of the plasma boundary $\partial P$, 
  all functions considered in this appendix may be assumed to be $\mathscr{C}^\infty$.
    \begin{lemma}
        Let $g$ be the normal magnetic field on $\partial P$ (i.e., the normal component of $B|_{\partial P}$). Then $g$ and $I_p$ determine completely the magnetic field $B$ in $P$.
        Besides, there exists a constant $C>0$ such that for every other  magnetic field $\tilde{B}$ with the same total poloidal current, $|B-\tilde{B}|_{L^2(P,\R^3)}\leq C |g-\tilde{g}|_{L^2(\partial P)}$ where $\tilde{g}$ is the normal component of $\tilde{B}|_{\partial P}$.
    \end{lemma}
    Before going to the proof of this statement, we emphasize that the structure of the space $L^2(P,\R^3)$ is well understood and admits 
    a generalized Hodge decomposition ($P$ is not a closed manifold, thus part \ref{append:Hodge} does not apply) from which the lemma follows easily. Such a decomposition is proved, for example, in \cite{Cantarella2002}. For completeness, we provide the following proof.
\begin{proof}
    We have the  cochain complex (meaning $\Ima(\grad) \subset \ker(\curl)$ and $\Ima(\curl) \subset \ker(\dive)$ )
    \[
    \xymatrix{
      \mathscr{C}^{\infty}(P) \ar[r]^-{\grad} & \mathscr{C}^{\infty}(P,\R^3) \ar[r]^-{\curl} & \mathscr{C}^{\infty}(P,\R^3) \ar[r]^-{\dive} & \mathscr{C}^{\infty}(P).
    }
    \]
    For simply connected domains of $\R^3$, the complex is an exact sequence, meaning that $\Ima(\grad)= \ker(\curl)$ and $\Ima(\curl) = \ker(\dive)$.
    For a 3D toroidal domain, the dimension of the quotient space $ \frac{\ker(\curl)}{\Ima(\grad)}$ is always one.
    This is a consequence of the De Rham cohomology of $P$. We refer to \cite[Diagram 16.15]{leeIntroductionSmoothManifolds2012} for further details.
    
    Thus $\ker(\curl) \not \subset \Ima(\grad)$, i.e., there exists
    $X\in \mathscr{C}^\infty(P,\R^3)$ such that $X \not \in \Ima(\grad)$ and $\curl X=0$.
    Without loss of generality, we can suppose that $\dive X=0$.
    Indeed, it is enough to consider $X'=X-\grad \zeta$ with $\zeta$ solution of the Poisson equation
   $$ \begin{array}{ll}
        \Delta \zeta=\dive X & \mbox{in } P, \\
        \zeta=0 & \mbox{on } \partial P.
    \end{array}
    $$
    To have an intuition, the reader can think of the vector field $X=\frac{e_{\theta}}{R}$ in $\R^3\setminus\{R=0\}$ in cylindrical coordinates $(R,\theta,z)$. This vector field is divergence and curl free but is not in the image of a gradient.

    We recall Maxwell's equations for a the static magnetic field in vacuum:
  \begin{align}
        \curl B =0 & \quad \mbox{in } P, \label{eq:MA}\\
        \dive B=0 & \quad \mbox{in } P. \label{eq:MT}
    \end{align}
    Equation \eqref{eq:MA} implies that there exist a scalar potential $\xi\in \mathscr{C}^\infty(P)$ and $\alpha \in \R$ such that
    $$B= \grad \xi +\alpha X.$$
    Using Stoke's theorem, the line integral of $B$ along $\Gamma$ is given by the total flux $I_p$ of electric currents across
  any surface enclosed by   
     $\Gamma$. In particular the contribution of the term $\nabla \xi$ to $I_p$ is zero, yielding
$$
    I_p=\oint_\Gamma \langle B,\frac{\Gamma'}{|\Gamma'|}\rangle d\mu_\Gamma =\oint_\Gamma\langle (\grad \xi+\alpha X), \frac{\Gamma'}{|\Gamma'|}\rangle d\mu_\Gamma=\alpha \oint_\Gamma \langle X, \frac{\Gamma'}{|\Gamma'|}\rangle d\mu_\Gamma.
$$
    The quantity $\oint_\Gamma \langle X,\frac{\Gamma'}{|\Gamma'|}\rangle d\mu_\Gamma$ is nonzero, since otherwise, by the De Rham isomorphism, $X$ would be in $\Ima{\nabla}$.
    Thus, $\alpha$ is uniquely determined by $I_p$, since $X$ does not depend on $B$.

    Equation \eqref{eq:MT} together with the normal component of $B$ on $\partial P$ give
  \begin{equation}         \label{eq:laplace_plasma}
  \begin{split}
        \Delta \xi=0 &\quad \text{in }P,\\
        \partial_n \xi =g - \alpha \langle X, n \rangle  &\quad \text{on }\partial P.
    \end{split}
\end{equation}
    Thus $\xi$ is determined by $g$ and $\alpha$ as the unique solution of a Laplace equation with Neumann boundary conditions.

    Finally, let $B=\nabla \xi + \alpha X$ and $\tilde B=\nabla {\tilde \xi} + \alpha X$ with $\xi$ and $\tilde \xi$ 
    the solutions of equation \eqref{eq:laplace_plasma} corresponding to $g$ and $\tilde g$, respectively. The difference $\delta =\xi-\tilde \xi$ is solution of
    \begin{align*}
        \Delta \delta =0 & \quad \text{in }P,\\
        \partial_n \delta =g-\tilde g & \quad \text{on }\partial P.
    \end{align*}
    By well-posedness of the Laplace equation with Neumann boundary conditions, there exists a constant $C(\partial P)$ such that $|\nabla \delta|_{H^{1/2}} \leq C(\partial P) |g-\tilde g|_{L^2(\partial P)}$.
    Thus, 
    $$|B-\tilde B|_{L^2(P,\R^3)}\lesssim |g-\tilde g|_{L^2(\partial P)},$$
concluding the proof. \end{proof}

\section{Reach constraint and sets of positive reach}\label{sec:append:reach}

In this section, we gather some reminders about the notion of reach. We refer to \cite[Chapter~6, Section~6]{delfour_shapes_2011} for more exhaustive explanations around this notion.

Recall first that, if $V$ is a nonempty subset of $\R^n$, its {\it skeleton}, denoted by $\operatorname{Sk}(V)$, is the set of all points in $\R^n$ whose projection onto $V$ is not unique. 
The set $V$ is said to have {\it a positive reach} whenever there exists $h>0$ such that 
\begin{equation}\label{TNUhV}
\text{every point $v$ of the tubular neighborhood $U_h(V)$ has
a unique projection point on $V$. }
\end{equation}
Recall that the definition of $U_h(V)$ is provided in Section~\ref{sec:notations}.
One thus defines the reach of $V$ as
$$
\operatorname{Reach}(V)=\sup \{h>0\mid \eqref{TNUhV}\text{ is satisfied}\}.
$$
An equivalent definition of the reach writes
$$
\operatorname{Reach}(V)=\inf \{\operatorname{Reach}(V,v)\mid v\in V\},
$$
where 
$$
\operatorname{Reach}(V,v)=\left\{\begin{array}{ll}
0 & \text{if }v\in \partial\overline{V}\cap \overline{\operatorname{Sk}(V)}\\
\sup\{h>0\mid \operatorname{Sk}(V)\cap B_h(v)=\emptyset\} & \text{otherwise,}
\end{array}
\right.
$$
where $B_h(v)$ denotes the Euclidean open ball centered at $v$ with radius $h$.

The notion of reach is actually closely related to the so-called uniform 
 ball condition. The next result make this relationship precise.

\begin{theorem}[Theorems 2.6 and 2.7 in \cite{dalphin_uniform_2018}]\label{theo:reachLit}
Let $\Omega$ be an open subset of $\R^n$ with a nonempty boundary.
\begin{itemize}
\item If there exists $h>0$ such that $\Omega$ satisfies a uniform ball condition, namely
\begin{equation}\label{ballCion}
\forall x\in \partial\Omega, \ \exists d_x\in \R^n \mid \Vert d_x\Vert_{\R^n}=1, \ B_h(x-hd_x)\subset \Omega \text{ and }B_h(x+hd_x)\subset \R^n\backslash \Omega,
\end{equation}
then $\partial\Omega$ has a positive reach which is larger than $h$ and the Lebesgue measure of $\partial\Omega$ in $\R^n$ is equal to 0. Furthermore, $\partial\Omega$ is a $\mathscr{C}^{1,1}$ hypersurface of $\R^n$.
\item If $\partial\Omega$ is a nonempty compact $\mathscr{C}^{1,1}$-hypersurface of $\R^n$, then there exists $h > 0$ such that $\Omega$ satisfies \eqref{ballCion}.
\item If $\partial\Omega$ has a positive reach and if its Lebesgue measure in $\R^n$ is equal to 0, then it satisfies the ball condition \eqref{ballCion} for every $h\in (0, \operatorname{Reach}(\partial\Omega))$ and in particular, $\partial\Omega$ is a $\mathscr{C}^{1,1}$ hypersurface of $\R^n$.
\end{itemize}
\end{theorem}

\section{Jacobian determinant and changes of variables on manifolds}
\label{appendix:cov}
We recall here some basic results about integration on manifolds which can be found in \cite{abrahamManifoldsTensorAnalysis1988} or \cite{sternChangeVariablesInequalities2013} for example.
Let $M$ and $N$ be two compact Riemannian $n$-dimensional manifolds with volume forms $\mu_M$ and $\mu_N$.
Let $\varphi: M\to N$ be an orientation preserving diffeomorphism.
Then, for any $v \in \mathscr{C}^1(N)$,
$$\int_{N}vd\mu_N=\int_M d\varphi^* (v\mu_N) 
$$
Besides, there exists a 
function  $J(\mu_M,\mu_N)\varphi$ on $M$, 
called the \emph{Jacobian determinant}, such that
$\varphi^*\mu_N=[J(\mu_M,\mu_N)\varphi]\mu_M$.
This implies the well-known change of variable formula
\begin{equation}\label{cdvIntegralSurf}
\int_N v d\mu_N=\int_M (v\circ \varphi) [J(\mu_M ,\mu_N)\varphi]d\mu_M.
\end{equation}

In the particular, when $M$ and $N$ are closed 2-dimensional submanifolds of $\R^3$ of class $\mathscr{C}^{1,1}$, and 
$\varphi$ is of the type $\varphi=\operatorname{Id}+\theta$ with $\theta\in W^{2,\infty}(\R^3, \R^3)$ 
and  $\Vert \theta\Vert_{W^{2,\infty}(\R^3, \R^3)}<1$ (so that $\varphi$ defines a diffeomorphism in $W^{2,\infty}(\R^3, \R^3)$), one has
\begin{equation*}
J(\mu_M ,\mu_N)\varphi=\det (\operatorname{Id}+D\theta)| ((\operatorname{Id}+D\theta)^{\top})^{-1}\nu|
\end{equation*}
with $\nu$ the outward normal to $M$. We refer for instance to \cite[Section 5.4.5]{henrotShapeVariationOptimization2018} for a shape optimization oriented proof or \cite[Chapter 5]{jostRiemannianGeometryGeometric2017} for a more differential geometry oriented presentation.


\section*{Acknowledgements} 
This work has been supported by the Inria AEX StellaCage. It was done in the framework of a collaboration between the Inria team CAGE and the startup \textit{Renaissance Fusion}\footnote{\url{https://stellarator.energy/}}. The authors would like to warmly thank Ugo Boscain, Chris Smiet, and Francesco Volpe for the numerous fascinating exchanges on 
the modeling of stellarators and their optimal design.

The first author were partially supported by the ANR Projects ``SHAPe Optimization - SHAPO'' and ``New TREnds in COntrol and Stabilization - TRECOS''. 

\bibliographystyle{abbrv} 
\bibliography{shape_opti}

\end{document}